\documentclass[platex]{amsart}
\usepackage{amssymb,amsmath,amsthm,empheq,comment,color,url}
\usepackage[utf8]{inputenc}

\usepackage[top=30truemm,bottom=30truemm,left=20truemm,right=20truemm]{geometry}

\theoremstyle{definition}
\newtheorem{theorem}{Theorem}[section]
\newtheorem{proposition}[theorem]{Proposition}
\newtheorem{definition}[theorem]{Definition}
\newtheorem{lemma}[theorem]{Lemma}
\newtheorem{corollary}[theorem]{Corollary}

\newtheorem*{remark*}{Remark}

\title{Minimal mass blow-up solutions for nonlinear Schr\"{o}dinger equations with an inverse potential}
\author[N. Matsui]{Naoki Matsui}
\date{\today}
\address[N. Mastui]{Department of Mathematics\\ Tokyo University of Science\\ 1-3 Kagurazaka, Shinjuku-ku, Tokyo 162-8601, Japan}
\email[N. Matsui]{1120703@ed.tus.ac.jp}


\providecommand{\keywords}[1]
{
  \small	
  \textbf{\textit{Keywords---}} #1
}

\DeclareMathOperator{\re}{Re}
\DeclareMathOperator{\im}{Im}

\DeclareMathOperator{\Span}{span}
\DeclareMathOperator{\Mod}{Mod}
\DeclareMathOperator{\Modop}{Mod_{op}}

\begin{document}
\maketitle

\begin{abstract}
We consider the following nonlinear Schr\"{o}dinger equation with an inverse potential:
\[
i\frac{\partial u}{\partial t}+\Delta u+|u|^{\frac{4}{N}}u\pm\frac{1}{|x|^{2\sigma}}u=0
\]
in $\mathbb{R}^N$. From the classical argument, the solution with subcritical mass ($\|u\|_2<\|Q\|_2$) is global and bounded in $H^1(\mathbb{R}^N)$. Here, $Q$ is the ground state of the mass-critical problem. Therefore, we are interested in the existence and behaviour of blow-up solutions for the threshold ($\left\|u_0\right\|_2=\left\|Q\right\|_2$). Previous studies investigate the existence and behaviour of the critical-mass blow-up solution when the potential is smooth or unbounded but algebraically tractable. There exist no results when classical methods can not be used, such as the inverse power type potential. However, we construct a critical-mass initial value for which the corresponding solution blows up in finite time. Moreover, we show that the corresponding blow-up solution converges to a certain blow-up profile in virial space.
\end{abstract}

\section{Introduction}
We consider the following nonlinear Schr\"{o}dinger equation with an inverse potential:
\begin{empheq}[left={(\mathrm{NLS\pm})\ \empheqlbrace\ }]{align*}
&i\frac{\partial u}{\partial t}+\Delta u+|u|^{\frac{4}{N}}u\pm\frac{1}{|x|^{2\sigma}}u=0,\\
&u(t_0)=u_0\nonumber
\end{empheq} 
in $\mathbb{R}^N$, where
\begin{eqnarray}
\label{index1}
\sigma\in\left(0,\min\left\{\frac{N}{2},1\right\}\right).
\end{eqnarray}
Then, (NLS$\pm$) is locally well-posed in $H^1(\mathbb{R}^N)$ (e.g., see \cite{CSSE}). This means that for any $u_0\in H^1(\mathbb{R}^N)$, there exists a unique maximal solution $u\in C^1((-T_*,T^*),H^{-1}(\mathbb{R}^N))\cap C((-T_*,T^*),H^1(\mathbb{R}^N))$. Moreover, the mass (i.e., $L^2$-norm) and energy $E$ of the solution are conserved by the flow, where 
\[
E(u):=\frac{1}{2}\left\|\nabla u\right\|_2^2-\frac{1}{2+\frac{4}{N}}\left\|u\right\|_{2+\frac{4}{N}}^{2+\frac{4}{N}}\mp\frac{1}{2}\||x|^{-\sigma}u\|_2^2.
\]
Furthermore, there is a blow-up alternative
\[
T^*<\infty\ \Rightarrow\ \lim_{t\nearrow T^*}\left\|\nabla u(t)\right\|_2^2=\infty.
\]

\subsection{Critical problem}
Firstly, we describe the results regarding the mass-critical problem:
\begin{empheq}[left={(\mathrm{CNLS})\ \empheqlbrace\ }]{align*}
&i\frac{\partial u}{\partial t}+\Delta u+|u|^{\frac{4}{N}}u=0,\\
&u(t_0)=u_0.
\end{empheq}
In particular, (NLS$\pm$) with $\sigma=0$ is attributed to (CNLS).

According to a classical variational argument (\cite{WGS}), there exists a unique classical solution of
\[
-\Delta Q+Q-\left|Q\right|^{\frac{4}{N}}Q=0,\quad Q\in H^1(\mathbb{R}^N),\quad Q>0,\quad Q\mathrm{\ is\ radial}
\]
(see \cite{BLGS,KGS}) which is called the ground state. For $u\in H^1(\mathbb{R}^N)$, if $\|u\|_2=\|Q\|_2$ ($\|u\|_2<\|Q\|_2$, $\|u\|_2>\|Q\|_2$), we say that $u$ has a critical mass (subcritical mass, supercritical mass, respectively). Here, $E_{\mathrm{crit}}(Q)=0$ holds, where $E_{\mathrm{crit}}$ is the critical energy. Moreover, the ground state $Q$ attains the optimal constant for the Gagliardo-Nirenberg inequality
\[
\left\|u\right\|_{2+\frac{4}{N}}^{2+\frac{4}{N}}\leq\left(1+\frac{2}{N}\right)\left(\frac{\left\|u\right\|_2}{\left\|Q\right\|_2}\right)^{\frac{4}{N}}\left\|\nabla u\right\|_2^2.
\]
Therefore, for any $u\in H^1(\mathbb{R}^N)$,
\[
E_{\mathrm{crit}}(u)\geq \frac{1}{2}\left\|\nabla u\right\|_2^2\left(1-\left(\frac{\left\|u\right\|_2}{\left\|Q\right\|_2}\right)^{\frac{4}{N}}\right).
\]
This inequality means that for any initial value with subcritical mass, the corresponding solution for (NLS) is global and bounded in $H^1(\mathbb{R}^N)$.

Regarding critical mass, we consider
\[
S(t,x):=\frac{1}{\left|t\right|^\frac{N}{2}}Q\left(\frac{x}{t}\right)e^{-\frac{i}{t}}e^{i\frac{\left|x\right|^2}{4t}},
\]
which is the solitary wave solution $u(t,x)=Q(x)e^{it}$ to which the pseudo-conformal transformation
\[
u(t,x)\ \mapsto\ \frac{1}{\left|t\right|^\frac{N}{2}}u\left(-\frac{1}{t},\pm\frac{x}{t}\right)e^{i\frac{\left|x\right|^2}{4t}}
\]
applied. Then, $S$ is also a solution for (CNLS) and
\[
\left\|S(t)\right\|_2=\left\|Q\right\|_2,\quad \left\|\nabla S(t)\right\|_2\sim\frac{1}{\left|t\right|}\quad (t\nearrow 0),
\]
meaning $S$ is a minimal-mass blow-up solution. Furthermore, up to the symmetries of the flow, the only critical-mass finite blow-up solution for (CNLS) is $S$ (\cite{MMMB}).

Regarding supercritical mass, there exists a solution for (CNLS) such that
\[
\left\|\nabla u(t)\right\|_2\sim\sqrt{\frac{\log\bigl|\log\left|T^*-t\right|\bigr|}{T^*-t}}\quad (t\nearrow T^*)
\]
(\cite{MRUPB,MRUDB}).

\subsection{Main results}
For (NLS$\pm$), it is immediately clear from the classical argument that if an initial value $u_0$ has a subcritical mass, then the corresponding solution is global and bounded in $H^1(\mathbb{R}^N)$.

In contrast, regarding critical mass in (NLS$+$), we obtain the following result:

\begin{theorem}[Existence of a minimal-mass blow-up solution]
\label{theorem:EMBS}
For any energy level $E_0\in\mathbb{R}$, there exist $t_0<0$ and a critical-mass radial initial value $u(t_0)\in \Sigma^1(\mathbb{R}^N)$ with $E(u_0)=E_0$ such that the corresponding solution $u$ for (NLS$+$) blows up at $T^*=0$. Moreover,
\[
\left\|u(t)-\frac{1}{\lambda(t)^\frac{N}{2}}P\left(t,\frac{x}{\lambda(t)}\right)e^{-i\frac{b(t)}{4}\frac{|x|^2}{\lambda(t)^2}+i\gamma(t)}\right\|_{\Sigma^1}\rightarrow 0\quad (t\nearrow 0)
\]
holds for some blow-up profile $P$, positive constants $C_1(\sigma)$ and $C_2(\sigma)$, positive-valued $C^1$ function $\lambda$, and real-valued $C^1$ functions $b$ and $\gamma$ such that
\[
P(t)\rightarrow Q\mbox{ in }H^1(\mathbb{R}^N),\quad\lambda(t)=C_1(\sigma)|t|^{\frac{1}{1+\sigma}}\left(1+o(1)\right),\quad b(t)=C_2(\sigma)|t|^{\frac{1-\sigma}{1+\sigma}}\left(1+o(1)\right),\quad \gamma(t)^{-1}=O\left(|t|^{\frac{1-\sigma}{1+\sigma}}\right)
\]
as $t\nearrow 0$.
\end{theorem}

Here, $\Sigma^1$ is defined as
\[
\Sigma^1:=\left\{u\in H^1\left(\mathbb{R}^N\right)\ \middle|\ xu\in L^2\left(\mathbb{R}^N\right)\right\}.
\]

On the other hands, the following results hold in (NLS$-$).

\begin{theorem}[Non-existence of a radial minimal-mass blow-up solution]
\label{theorem:NEMBS}
Assume $N\geq 2$. If $u_0\in H^1_\mathrm{rad}(\mathbb{R}^N)$ such that $\|u_0\|_2=\|Q\|_2$, the corresponding solution $u$ for (NLS$-$) is global and bounded in $H^1(\mathbb{R}^N)$.
\end{theorem}

See Appendix \ref{ProofESCBS} for the proof.

\begin{theorem}
\label{theorem:ESCBS}
For any $\delta>0$, there exists $u_0\in\Sigma^2$ such that $\|u_0\|_2=\|Q\|_2+\delta$ and the corresponding solution $u$ for (NLS$-$) blows up at finite time.
\end{theorem}

This is a consequence of \cite{N}.

\subsection{Outline of proof}
We will now outline the proof for Theorem \ref{theorem:EMBS}.

In Section \ref{sec:Preliminaries}, we describe some basic statements that are used in the proof of Theorem \ref{theorem:EMBS}.

In Section \ref{sec:constprof} (and Appendix \ref{sec:solofs}), we construct a blow-up profile and introduce the \textit{decomposition} of functions.

From Section \ref{sec:uniesti} to Section \ref{sec:proof}, we prove Theorem \ref{theorem:EMBS} using the technique described in Le Coz-Martel-Rapha\"{e}l \cite{LMR} and Martel-Szeftel \cite{RSEU}.

In Section \ref{sec:uniesti}, we set an initial value and decompose the corresponding solution for (NLS$+$) into a core part and remainder part. By rescaling the time variable, we consider an equation for the remainder part in rescaled time and estimate the modulation equations of the parameters for decomposition.

In Section \ref{sec:MEF}, by using the coercivity of the linearised Schr\"{o}dinger operator, we estimate the energy of the remainder part.

In Section \ref{sec:bootstrap}, by using bootstrapping, we justify the arguments in Sections \ref{sec:MEF}.

In Section \ref{sec:convesti}, we restore the time variable.

In Section \ref{sec:proof}, we complete the proof of Theorem \ref{theorem:EMBS}.

\subsection{Previous results}
We describe previous results regarding  the following nonlinear Schr\"{o}dinger equation with a potential:
\begin{empheq}[left={(\mathrm{PNLS})\ \empheqlbrace\ }]{align*}
&i\frac{\partial u}{\partial t}+\Delta u+|u|^{\frac{4}{N}}u+Vu=0,\\
&u(t_0)=u_0\nonumber
\end{empheq} 
in $\mathbb{R}^N$.

\begin{theorem}[\cite{N}]
\label{theorem:N}
We assume that $V\in \left(L^p(\mathbb{R}^N)+L^\infty(\mathbb{R}^N)\right)\cap C^1(\mathbb{R}^N)$ for some $p\in[1,\infty]\cap(\frac{N}{2},\infty]$ and $\nabla V\in L^q(\mathbb{R}^N)+L^\infty(\mathbb{R}^N)$ for some $q\in[2,\infty]\cap(N,\infty]$. Then, there exist $t_0<0$ and a critical-mass radial initial value $u(t_0)\in \Sigma^1(\mathbb{R}^N)$ such that the corresponding solution $u$ for (PNLS) blows up at $T^*=0$. Moreover,
\[
\left\|u(t)-\frac{1}{\lambda(t)^\frac{N}{2}}Q\left(\frac{x+w(t)}{\lambda(t)}\right)e^{-i\frac{b(t)}{4}\frac{|x+w(t)|^2}{\lambda(t)^2}+i\gamma(t)}\right\|_{\Sigma^1}\rightarrow 0\quad (t\nearrow 0)
\]
holds for some positive-valued $C^1$ function $\lambda$, real-valued $C^1$ functions $b$ and $\gamma$, and $\mathbb{R}^N$-valued $C^1$ function $w$ such that
\[
\lambda(t)=|t|\left(1+o(1)\right),\quad b(t)=|t|\left(1+o(1)\right),\quad \gamma(t)^{-1}=O\left(|t|^{-1}\right),\quad |w(t)|=o(|t|)
\]
as $t\nearrow 0$.
\end{theorem}

\begin{theorem}[Carles \cite{C}]
\label{theorem:C}
If $V=E\cdot x$ for some $E\in\mathbb{R}^N$, then (PNLS) has a finite time blow-up solution
\begin{eqnarray}
S(t,x):=\frac{1}{|t|^\frac{N}{2}}Q\left(\frac{x+t^2E}{t}\right)e^{i\left(\frac{|x+t^2E|^2}{4t}-\frac{1}{t}-\sqrt{2}tE\cdot x+\frac{t^3}{3}|E|^2\right)}.
\end{eqnarray}
In particular, $\|S\|_2=\|Q\|_2$.
\end{theorem}

\begin{theorem}[Carles and Nakamura \cite{CN}]
\label{theorem:CN}
If $V=\omega|x|^2$ for some $\omega>0$, then (PNLS) has a finite time blow-up solution
\begin{eqnarray}
S(t,x):=\frac{1}{\left|\frac{2}{\omega}\sinh \left(\frac{\omega t}{2}\right)\right|^\frac{N}{2}}Q\left(\frac{\omega x}{2\sinh \left(\frac{\omega t}{2}\right)}\right)e^{i\left(\frac{\omega|x|^2}{8\sinh \left(\frac{\omega t}{2}\right)\cosh \left(\frac{\omega t}{2}\right)}-\frac{\omega}{2\sinh \left(\frac{\omega t}{2}\right)}+\frac{\omega}{4}|x|^2\tanh \left(\frac{\omega t}{2}\right)\right)}.
\end{eqnarray}
In particular, $\|S\|_2=\|Q\|_2$.
\end{theorem}

\begin{theorem}[E. Csobo and F. Genoud \cite{EF}]
\label{theorem:EF}
Let $N\geq 3$ and $V=\frac{c}{|x|^2}$ for some $0<c<\frac{(N-2)^2}{4}$. Then, (PNLS) has a finite time blow-up solution
\[
S(t,x):=\left(\frac{\lambda_0}{T-t}\right)^\frac{N}{2}\tilde{Q}\left(\frac{\lambda_0x}{T-t}\right)e^{-i\frac{|x|^2}{4(T-t)}+i\frac{{\lambda_0}^2}{T-t}+i\gamma_0},
\]
where $T,\gamma_0\in\mathbb{R}$, $\lambda_0>0$, and $\tilde{Q}$ is a unique radial positive classical solution of
\[
-\Delta \varphi+\frac{c}{|x|^2}\varphi-\varphi+|\varphi|^\frac{4}{N}\varphi=0.
\]
Moreover, $S$ is a minimal-mass blow-up solution.
\end{theorem}

Finally, we introduce the result of Le Coz, Martel, and Rapha\"{e}l \cite{LMR} based on the methodology of seminal work Martel and Szeftel \cite{RSEU} for
\begin{empheq}[left={(\mathrm{DPNLS})\ \empheqlbrace\ }]{align*}
&i\frac{\partial u}{\partial t}+\Delta u+|u|^{\frac{4}{N}}u+\epsilon|u|^{p-1}u=0,\\
&u(0)=u_0.\nonumber
\end{empheq}

\begin{theorem}[Le Coz, Martel, and Rapha\"{e}l \cite{LMR}]
\label{theorem:LMR}
Let $N=1,2,3$, $1<p<1+\frac{4}{N}$, and $\epsilon=1$. Then, for any energy level $E_0\in\mathbb{R}$, there exist $t_0$ and a radially symmetric initial value $u_0(t_0)\in H^1(\mathbb{R}^N)$ with
\[
\|u(t)\|_2=\|Q\|_2,\quad E(u(t_0))=E_0
\]
such that the corresponding solution $u$ for (DPNLS) blows up at $T^*=0$ with a blow-up rate of
\[
\|\nabla u(t)\|_2=\frac{C(p)+o_{t\nearrow 0}(1)}{|t|^{\frac{4}{4+N(p-1)}}},
\]
where $C(p)>0$.
\end{theorem}

\begin{theorem}[\cite{LMR}]
Let $N=1,2,3$, $1<p<1+\frac{4}{N}$, and $\epsilon=-1$. If an initial value has critical mass, then the corresponding solution of (DPNLS) is global and bounded in $H^1(\mathbb{R}^N)$.
\end{theorem}

\subsection{Comments regarding the main results}
We present some comments regarding Theorem \ref{theorem:EMBS} below.

In Theorem \ref{theorem:C}, Theorem \ref{theorem:CN}, and Theorem \ref{theorem:EF}, the blow-up solutions are explicitly constructed by the transformation of a solitary wave. In contrast to these, the method used in Theorem \ref{theorem:EMBS} is not classical. In particular, Theorem \ref{theorem:EMBS} is the first result for a unbounded potential without algebraic properties.

In terms of blow-up rates, we have $|t|^{-\frac{1}{1+\sigma}}\rightarrow |t|^{-\frac{1}{2}}$ as $\sigma\rightarrow 1$. This blow-up rate is different from the Theorem \ref{theorem:EF}. This may be since (NLS$\pm$) is not locally well-posed in $H^1$ when $\sigma=1$. Moreover, since $C_1(\sigma)\rightarrow\infty$ as $\sigma\rightarrow 1$, the limit dose not make sense.

The potential in Theorem \ref{theorem:N} is smooth. However, the potential in Theorem \ref{theorem:EMBS} is singular at the origin. This difference reflect in the blow-up rate.

The method in Theorem \ref{theorem:EMBS} could also be applied to nonlinear terms of the form $|x|^{-2\sigma}|u|^{p-1}u$.

\subsection{Notations}
In this section, we introduce the notation used in this paper.

Let
\[
\mathbb{N}:=\mathbb{Z}_{\geq 1},\quad\mathbb{N}_0:=\mathbb{Z}_{\geq 0}.
\]
Unless otherwise noted, we define
\begin{align*}
&(u,v)_2:=\re\int_{\mathbb{R}^N}u(x)\overline{v}(x)dx,\quad \left\|u\right\|_p:=\left(\int_{\mathbb{R}^N}|u(x)|^pdx\right)^{\frac{1}{p}},\quad f(z):=|z|^{\frac{4}{N}}z,\quad  F(z):=\frac{1}{2+\frac{4}{N}}|z|^{2+\frac{4}{N}}.
\end{align*}
By identifying $\mathbb{C}$ with $\mathbb{R}^2$, we denote the differentials of $f$ and $F$ by $df$ and $dF$, respectively. We define
\[
\Lambda:=\frac{N}{2}+x\cdot\nabla,\quad L_+:=-\Delta+1-\left(1+\frac{4}{N}\right)Q^{\frac{4}{N}},\quad L_-:=-\Delta+1-Q^{\frac{4}{N}}.
\]
Then,
\[
L_-Q=0,\quad L_+\left(\Lambda Q\right)=-2Q,\quad L_-\left(|x|^2Q\right)=-4\Lambda Q,\quad L_+\rho=|x|^2 Q
\]
hold, where $\rho$ is the unique radial Schwartz solution of $L_+\rho=|x|^2 Q$. Furthermore, there exists $\mu>0$ such that
\[
\forall u\in H_{\mathrm{rad}}^1(\mathbb{R}^N),\quad \left\langle L_+\re u,\re u\right\rangle+\left\langle L_-\im u,\im u\right\rangle\geq \mu\left\|u\right\|_{H^1}^2-\frac{1}{\mu}\left((\re u,Q)_2^2+(\re u,|x|^2 Q)_2^2+(\im u,\rho)_2^2\right)
\]
(e.g., see \cite{MRO,MRUPB,RSEU,WL}). We introduce
\[
\Sigma^m:=\left\{u\in H^m(\mathbb{R}^N)\ \middle|\ |x|^m u\in L^2(\mathbb{R}^N)\right\}.
\]
Additionally, we denote by $\mathcal{Y}$ the set of functions $g\in C^{\infty}(\mathbb{R}^N\setminus\{0\})\cap C(\mathbb{R}^N)\cap H^1_{\mathrm{rad}}(\mathbb{R}^N)$ such that
\[
\forall\alpha\in{\mathbb{N}_0}^N\exists C_\alpha,\kappa_\alpha>0,\ |x|\geq 1\Rightarrow \left|\left(\frac{\partial}{\partial x}\right)^\alpha g(x)\right|\leq C_\alpha(1+|x|)^{\kappa_{\alpha}}Q(x)
\]
and by $\mathcal{Y}'$ the set of functions $g\in\mathcal{Y}$ such that
\[
\Lambda g\in H^1(\mathbb{R}^N)\cap C(\mathbb{R}^N).
\]
Finally, we use $\lesssim$ and $\gtrsim$ when the inequalities hold except for non-essential positive constant differences and $\approx$ when $\lesssim$ and $\gtrsim$ hold.

\section{Preliminaries}
\label{sec:Preliminaries}
We provide the following statements regarding notations without proofs.

\begin{proposition}
\label{GSP}
For any $\alpha\in{\mathbb{N}_0}^N$, there exists a constant $C_\alpha>0$ such that $\left|\left(\frac{\partial}{\partial x}\right)^\alpha Q(x)\right|\leq C_\alpha Q(x)$. Similarly, $\left|\left(\frac{\partial}{\partial x}\right)^\alpha \rho(x)\right|\leq C_\alpha(1+|x|)^{\kappa_\alpha} Q(x)$ holds (e.g., \cite{LMR,N}).
\end{proposition}

\begin{lemma}
For the ground state $Q$, 
\[
(Q,\rho)_2=\frac{1}{2}\bigl\||x|^2Q\bigr\|_2^2
\]
holds.
\end{lemma}

\begin{lemma}
For an appropriate function $w$,
\[
\left(|x|^{2p}w,\Lambda w\right)_2=-p\bigl\||x|^pw\bigr\|_2^2,\quad (-\Delta w,\Lambda w)_2=\bigl\|\nabla w\bigr\|_2^2,\quad (|w|^pw,\Lambda w)_2=\frac{Np}{2(p+2)}\|w\|_{p+2}^{p+2}
\]
hold.
\end{lemma}

\begin{lemma}[Properties of $F$ and $f$]
\label{Fdef}
For $F$ and $f$,
\begin{align*}
&\frac{\partial F}{\partial\re}=\re f,\quad \frac{\partial F}{\partial\im}=\im f,\quad \frac{\partial \re f}{\partial\im}=\frac{\partial \im f}{\partial\re},\\
&\frac{\partial}{\partial s}F(z(s))=f(z(s))\cdot\frac{\partial z}{\partial s}=\re\left(f(z(s))\overline{\frac{\partial z}{\partial s}}\right),\\
&dF(z)(w)=f(z)\cdot w=\re\left(f(z)\overline{w}\right),\\
&df(z)(w_1)\cdot w_2=df(z)(w_2)\cdot w_1,\\
&\frac{\partial}{\partial s}dF(z(s))(w(s))=df(z(s))(w(s))\cdot\frac{\partial z}{\partial s}+f(z(s))\cdot\frac{\partial w}{\partial s},\\
&\frac{\partial}{\partial w}\int_{\mathbb{R}^N}\left(F(z(x)+w(x))-F(z(x))-dF(z(x))(w(x))\right)dx=f(z+w)-f(z),\\
&L_+\left(\re Z\right)+iL_-\left(\im Z\right)=-\Delta Z+Z-df(Q)(Z)
\end{align*}
hold. When identifying $\mathbb{C}$ with $\mathbb{R}^2$, $\cdot$ is the inner product of $\mathbb{R}^2$.
\end{lemma}

\section{Construction of a blow-up profile}
\label{sec:constprof}
In this section, we construct a blow-up profile $P$ and introduce a decomposition of functions.

For $K\in\mathbb{N}$, we define
\[
\Sigma_K:=\left\{\ (j,k)\in{\mathbb{N}_0}^2\ \middle|\ j+k\leq K\ \right\}.
\]

\begin{proposition}
\label{theorem:constprof}
Let $K,K'\in\mathbb{N}$ be sufficiently large. Let $\lambda(s)>0$ and $b(s)\in\mathbb{R}$ be $C^1$ functions of $s$ such that $\lambda(s)+|b(s)|\ll 1$.

(i) \textit{Existence of blow-up profile.} For any $(j,k)\in\Sigma_{K+K'}$, there exist real-valued functions $P_{j,k}^+,P_{j,k}^-\in\mathcal{Y}'$ and $\beta_{j,k}\in\mathbb{R}$ such that $P$ satisfies
\[
i\frac{\partial P}{\partial s}+\Delta P-P+f(P)+\lambda^\alpha \frac{1}{|y|^{2\sigma}}P+\theta\frac{|y|^2}{4}P=\Psi,
\]
where $\alpha=2-2\sigma$, and  $P$ and $\theta$ are defined by
\begin{align*}
P(s,y)&:=Q(y)+\sum_{(j,k)\in\Sigma_{K+K'}}\left(b(s)^{2j}\lambda(s)^{(k+1)\alpha}P_{j,k}^+(y)+ib(s)^{2j+1}\lambda(s)^{(k+1)\alpha}P_{j,k}^-(y)\right),\\
\theta(s)&:=\sum_{(j,k)\in\Sigma_{K+K'}}b(s)^{2j}\lambda(s)^{(k+1)\alpha}\beta_{j,k}.
\end{align*}
Moreover, for some $\epsilon'>0$ which is sufficiently small,
\[
\left\|e^{\epsilon'|y|}\Psi\right\|_{H^1}\lesssim\lambda^\alpha\left(\left|b+\frac{1}{\lambda}\frac{\partial \lambda}{\partial s}\right|+\left|\frac{\partial b}{\partial s}+b^2-\theta\right|\right)+(b^2+\lambda^\alpha)^{K+2}
\]
holds.

(ii) \textit{Mass and energy properties of blow-up profile.} Let define
\[
P_{\lambda,b,\gamma}(s,x):=\frac{1}{\lambda(s)^\frac{N}{2}}P\left(s,\frac{x}{\lambda(s)}\right)e^{-i\frac{b(s)}{4}\frac{|x|^2}{\lambda(s)^2}+i\gamma(s)}.
\]
Then,
\begin{align*}
\left|\frac{d}{ds}\|P_{\lambda,b,\gamma}\|_2^2\right|&\lesssim\lambda^\alpha\left(\left|b+\frac{1}{\lambda}\frac{\partial \lambda}{\partial s}\right|+\left|\frac{\partial b}{\partial s}+b^2-\theta\right|\right)+(b^2+\lambda^\alpha)^{K+2},\\
\left|\frac{d}{ds}E(P_{\lambda,b,\gamma})\right|&\lesssim\frac{1}{\lambda^2}\left(\left|b+\frac{1}{\lambda}\frac{\partial \lambda}{\partial s}\right|+\left|\frac{\partial b}{\partial s}+b^2-\theta\right|+(b^2+\lambda^\alpha)^{K+2}\right)
\end{align*}
hold. Moreover,
\begin{eqnarray}
\label{Eesti}
\left|8E(P_{\lambda,b,\gamma})-\||\cdot|Q\|_2^2\left(\frac{b^2}{\lambda^2}-\frac{2\beta}{2-\alpha}\lambda^{\alpha-2}\right)\right|\lesssim\frac{\lambda^\alpha(b^2+\lambda^\alpha)}{\lambda^2}
\end{eqnarray}
holds, where 
\[
\beta:=\beta_{0,0}=\frac{4\sigma\||y|^{-\sigma}Q\|_2^2}{\||y|Q\|_2^2}.
\]
\end{proposition}

\begin{proof}
See \cite{LMR} for details of proofs.

We prove (i). We set
\[
Z:=\sum_{(j,k)\in\Sigma_{K+K'}}b^{2j}\lambda^{k\alpha}P_{j,k}^++i\sum_{(j,k)\in\Sigma_{K+K'}}b^{2j+1}\lambda^{k\alpha}P_{j,k}^-.
\]
Then, $P=Q+\lambda^\alpha Z$ holds. Moreover, let set
\begin{align*}
\Theta(s)&:=\sum_{(j,k)\in\Sigma_{K+K'}}b(s)^{2j}\lambda(s)^{(k+1)\alpha}c^+_{j,k},\\
\Phi&:=i\frac{\partial P}{\partial s}+\Delta P-P+f(P)+\lambda^\alpha\frac{1}{|y|^{2\sigma}}P+\theta\frac{|y|^2}{4}P+\Theta Q,
\end{align*}
where $P_{j,k}^+,P_{j,k}^-\in\mathcal{Y}'$ and $\beta_{j,k},c^+_{j,k}\in\mathbb{R}$ are to be determined.

Firstly, we have
\begin{align*}
i\frac{\partial P}{\partial s}&=-i\sum_{(j,k)\in\Sigma_{K+K'}}((k+1)\alpha+2j)b^{2j+1}\lambda^{(k+1)\alpha}P_{j,k}^+\\
&\hspace{40pt}+i\sum_{j,k\geq 0}b^{2j+1}\lambda^{(k+1)\alpha}F_{j,k}^{\frac{\partial P}{\partial s},-}+\sum_{j,k\geq 0}b^{2j}\lambda^{(k+1)\alpha}F_{j,k}^{\frac{\partial P}{\partial s},+}+\Psi^\frac{\partial P}{\partial s},
\end{align*}
where
\begin{align*}
\Phi^\frac{\partial P}{\partial s}&=\left(b+\frac{1}{\lambda}\frac{\partial \lambda}{\partial s}\right)\sum_{(j,k)\in\Sigma_{K+K'}}(k+1)\alpha b^{2j}\lambda^{(k+1)\alpha}(iP_{j,k}^+-bP_{j,k}^-)\nonumber\\
&\hspace{40pt}+\left(\frac{\partial b}{\partial s}+b^2-\theta\right)\sum_{(j,k)\in\Sigma_{K+K'}}b^{2j-1}\lambda^{(k+1)\alpha}(2jiP_{j,k}^+-(2j+1)bP_{j,k}^-)
\end{align*}
and for $j,k\geq0$, $F_{j,k}^{\frac{\partial P}{\partial s},\pm}$ consists of $P_{j',k'}^\pm$ and $\beta_{j',k'}$ for $(j',k')\in\Sigma_{K+K'}$ such that $k'\leq k-1$ and $j'\leq j+1$ or $k'\leq k$ and $j'\leq j-1$. Only a finite number of these functions are non-zero. In particular, $F_{j,k}^{\frac{\partial P}{\partial s},\pm}$ belongs to $\mathcal{Y}'$ and $F_{0,0}^{\frac{\partial P}{\partial s},\pm}$=0.

Next, we have
\begin{align*}
\Delta P-P+|P|^\frac{4}{N}P=&-\sum_{(j,k)\in\Sigma_{K+K'}}b^{2j}\lambda^{(k+1)\alpha}L_+P_{j,k}^+-i\sum_{(j,k)\in\Sigma_{K+K'}}b^{2j+1}\lambda^{(k+1)\alpha}L_-P_{j,k}^-\\
&\hspace{40pt}+\sum_{j,k\geq 0}b^{2j}\lambda^{(k+1)\alpha}F_{j,k}^{f,+}+i\sum_{j,k\geq 0}b^{2j+1}\lambda^{(k+1)\alpha}F_{j,k}^{f,-}+\Phi^f,
\end{align*}
where
\[
\Phi^f=f(Q+\lambda^\alpha Z)-\sum_{k=0}^{K+K'+1}\frac{1}{k!}d^kf(Q)(\lambda^\alpha Z,\cdots,\lambda^\alpha Z)
\]
and for $j,k\geq0$, $F_{j,k}^{f,\pm}$ consists of $Q$, $P_{j',k'}^\pm$, and $\beta_{j',k'}$ for $(j',k')\in\Sigma_{K+K'}$ such that $k'\leq k-1$ and $j'\leq j$. Only a finite number of these functions are non-zero. In particular, $F_{j,k}^{f,\pm}$ belongs to $\mathcal{Y}'$ and $F_{0,0}^{f,\pm}$=0.

Next, we have
\[
\lambda^\alpha \frac{1}{|y|^{2\sigma}}P=\sum_{j+k\geq0}\left(b^{2j}\lambda^{(k+1)\alpha}\frac{1}{|y|^{2\sigma}}F_{j,k}^{\sigma,+}+ib^{2j}\lambda^{(k+1)\alpha}\frac{1}{|y|^{2\sigma}}F_{j,k}^{\sigma,-}\right),
\]
where
\[
F_{j,k}^{\sigma,+}=\left\{
\begin{array}{ll}
Q&(j=k=0)\\
0&(j\geq1,\ k=0)\\
P_{j,k-1}^+&(k\geq 1)
\end{array}
\right.,\quad F_{j,k}^{\sigma,-}=\left\{
\begin{array}{ll}
0&(k=0)\\
P_{j,k-1}^-&(k\geq 1)
\end{array}
\right..
\]

Finally, we have
\[
\theta\frac{|y|^2}{4}P=\sum_{(j,k)\in\Sigma_{K+K'}}b^{2j}\lambda^{(k+1)\alpha}\beta_{j,k}\frac{|y|^2}{4}Q+\sum_{j,k\geq 0}b^{2j}\lambda^{(k+1)\alpha}F_{j,k}^{\theta,+}+i\sum_{j,k\geq 0}b^{2j+1}\lambda^{(k+1)\alpha}F_{j,k}^{\theta,-}
\]
and for $j,k\geq0$, $F_{j,k}^{\theta,\pm}$ consists of $Q$, $P_{j',k'}^\pm$, and $\beta_{j',k'}$ for $(j',k')\in\Sigma_{K+K'}$ such that $k'\leq k-1$ and $j'\leq j$. Only a finite number of these functions are non-zero. In particular, $F_{j,k}^{\theta,\pm}$ belongs to $\mathcal{Y}'$ and $F_{0,0}^{\theta,\pm}$=0.

Here, we define
\begin{align*}
F_{j,k}^{\pm}&:=F_{0,0}^{\frac{\partial P}{\partial s},\pm}+F_{0,0}^{\theta,\pm},\\
\Phi^{>K+K'}&:=\sum_{(j,k)\not\in\Sigma_{K+K'}}b^{2j}\lambda^{(k+1)\alpha}F_{j,k}^++i\sum_{(j,k)\not\in\Sigma_{K+K'}}b^{2j+1}\lambda^{(k+1)\alpha}F_{j,k}^-,\\
\Phi&:=\Phi^\frac{\partial P}{\partial s}+\Phi^f+\Phi^{>K+K'}+\lambda^{(K+K'+2)\alpha}\frac{1}{|y|^{2\sigma}}P_{0,K+K'}^++ib\lambda^{(K+K'+2)\alpha}\frac{1}{|y|^{2\sigma}}P_{0,K+K'}^-.
\end{align*}
Then, $\Phi^{>K+K'}$ is a finite sum and we obtain
\begin{align*}
&i\frac{\partial P}{\partial s}+\Delta P-P+f(P)+\lambda^\alpha\frac{1}{|y|^{2\sigma}}P+\theta\frac{|y|^2}{4}P+\Theta Q\\
=&\sum_{(j,k)\in\Sigma_{K+K'}}b^{2j}\lambda^{(k+1)\alpha}\left(-L_+P_{j,k}^++\beta_{j,k}\frac{|y|^2}{4}Q+\frac{1}{|y|^{2\sigma}}F_{j,k}^{\sigma,+}+F_{j,k}^++c_{j,k}^+Q\right)\\
&\hspace{10pt}+i\sum_{(j,k)\in\Sigma_{K+K'}}b^{2j+1}\lambda^{(k+1)\alpha}\left(-L_-P_{j,k}^--((k+1)\alpha+2j)P_{j,k}^++\frac{1}{|y|^{2\sigma}}F_{j,k}^{\sigma,-}+F_{j,k}^-\right)\\
&\hspace{20pt}+\Phi.
\end{align*}
For each $(j,k)\in\Sigma_{K+K'}$, we choose recursively $P_{j,k}^\pm\in\mathcal{Y}'$ and $\beta_{j,k},c_{j,k}^+\in\mathbb{R}$ that are solutions of the systems
\begin{empheq}[left={(S_{j,k})\ \empheqlbrace\ }]{align*}
&L_+P_{j,k}^+-F_{j,k}^+-\beta_{j,k}\frac{|y|^2}{4}Q-\frac{1}{|y|^{2\sigma}}F_{j,k}^{\sigma,+}-c_{j,k}^+Q=0\\
&L_-P_{j,k}^--F_{j,k}^-+((k+1)\alpha+2j)P_{j,k}^+-\frac{1}{|y|^{2\sigma}}F_{j,k}^{\sigma,-}=0
\end{empheq}
and satisfy
\[
c_{j,k}^+=0\ (j+k\leq K),\quad \frac{1}{|y|^2}P_{j,k}^\pm,\frac{1}{|y|}|\nabla P_{j,k}^\pm|\in L^\infty(\mathbb{R}^N).
\]
See Appendix \ref{sec:solofs} for details.

In the same way as Proposition 2.1 in \cite{LMR}, for some $\epsilon'>0$ which is sufficiently small, we have
\begin{align*}
\left\|e^{\epsilon'|y|}\Phi^\frac{\partial P}{\partial s}\right\|_{H^1}&\lesssim\lambda^{\alpha}\left(\left|b+\frac{1}{\lambda}\frac{\partial \lambda}{\partial s}\right|+\left|\frac{\partial b}{\partial s}+b^2-\theta\right|\right),\\
\left\|e^{\epsilon'|y|}\Phi^f\right\|_{H^1}&\lesssim\lambda^{(K+K'+2)\alpha},\\
\left\|e^{\epsilon'|y|}\Phi^{>K+K'}\right\|_{H^1}&\lesssim\left(b^2+\lambda^{\alpha}\right)^{K+K'+2}.
\end{align*}
Moreover,
\[
\left\|e^{\epsilon'|y|}\Theta Q\right\|_{H^1}\lesssim \left(b^2+\lambda^{\alpha}\right)^{K+2}
\]
holds. Therefore, we have
\[
\left\|e^{\epsilon'|y|}\Psi\right\|_{H^1}\lesssim \lambda^{\alpha}\left(\left|b+\frac{1}{\lambda}\frac{\partial \lambda}{\partial s}\right|+\left|\frac{\partial b}{\partial s}+b^2-\theta\right|\right)+\left(b^2+\lambda^{\alpha}\right)^{K+2},
\]
where $\Psi:=\Phi-\Theta Q$.

Next, we prove only (\ref{Eesti}) of (ii). The rest is the same as in \cite{LMR}. We have
\begin{align*}
\lambda^2E(P_{\lambda,b,\gamma})=&\frac{1}{2}\left\|\nabla Q+\lambda^\alpha\nabla Z\right\|_2^2-\int_{\mathbb{R}^N}F(Q+\lambda^\alpha Z)dx-\frac{\lambda^\alpha}{2}\left\||y|^{-\sigma}Q+\lambda^\alpha|y|^{-\sigma}Z\right\|_2^2\\
&\hspace{20pt}-\frac{b}{2}(iQ+i\lambda^\alpha Z,\Lambda Q+\lambda^\alpha\Lambda Z)_2+\frac{b^2}{8}\left\||y|Q+\lambda^\alpha |y|Z\right\|_2^2.
\end{align*}
Here,
\begin{align*}
&\frac{1}{2}\left\|\nabla Q\right\|_2^2=\int_{\mathbb{R}^N}F(Q)dx,\quad(\nabla Q,\lambda^\alpha \nabla Z)_2=-(Q,\lambda^\alpha Z)_2+\int_{\mathbb{R}^N}dF(Q)(\lambda^\alpha Z)dx,\\
&\frac{1}{2}\left\||y|^{-\sigma}Q\right\|_2^2=\frac{1}{8}\left\||y|Q\right\|_2^2\frac{2\beta}{2-\alpha},\quad (iQ,\Lambda Q)_2=0
\end{align*}
hold and we have
\begin{align*}
(Q,\lambda^\alpha Z)_2&=\sum_{(j,k)\in\Sigma_{K+K'},\ j+k\geq 1}b^{2j}\lambda^{(k+1)\alpha}\left(Q,P_{j,k}^+\right)_2=O(\lambda^\alpha(b^2+\lambda^\alpha)),\\
b(i\lambda Z,\Lambda Q)_2&=-b\sum_{(j,k)\in\Sigma_{K+K'}}b^{2j+1}\lambda^{(k+1)\alpha}\left(P_{j,k}^-,\Lambda Q\right)_2=O(b^2\lambda^\alpha).
\end{align*}
Therefore, we have
\begin{align*}
\lambda^2\frac{d}{ds}E(P_{\lambda,b,\gamma})=&-\int_{\mathbb{R}^N}\left(F(Q+\lambda^\alpha Z)-F(Q)-dF(Q)(\lambda^\alpha Z)\right)dx\\
&\hspace{20pt}-\frac{\lambda^\alpha}{8}\left\||y|Q\right\|_2^2\frac{2\beta}{2-\alpha}+\frac{b^2}{8}\left\||y|Q\right\|_2^2+O(\lambda^\alpha(b^2+\lambda^\alpha))
\end{align*}
and
\[
\int_{\mathbb{R}^N}\left(F(Q+\lambda^\alpha Z)-F(Q)-dF(Q)(\lambda^\alpha Z)\right)dx=O(\lambda^{2\alpha}).
\]
Consequently, we have the conclusion.
\end{proof}

\begin{lemma}[Decomposition]
\label{decomposition}
There exist constants $\overline{l},\overline{\lambda},\overline{b},\overline{\gamma}>0$ such that the following logic holds.

Let $I$ be an interval, let $\delta>0$ be sufficiently small, and let $u\in C(I,H^1(\mathbb{R}^N))\cap C^1(I,H^{-1}(\mathbb{R}^N))$ satisfy that there exist functions $\lambda\in \mathrm{Map}(I,(0,\overline{l}))$ and $\gamma\in \mathrm{Map}(I,\mathbb{R})$ such that 
\[
\forall\ t\in I,\ \left\|\lambda(t)^{\frac{N}{2}}u(t,\lambda(t)y)e^{i\gamma(t)}-Q\right\|_{H^1}< \delta.
\]
Then, (given $\tilde{\gamma}(0)$) there exist unique functions $\tilde{\lambda}\in C^1(I,(0,\infty))$ and $\tilde{b},\tilde{\gamma}\in C^1(I,\mathbb{R})$ that are independent of $\lambda$ and $\gamma$ such that 
\begin{align}
\label{mod}
u(t,x)&=\frac{1}{\tilde{\lambda}(t)^{\frac{N}{2}}}\left(P+\tilde{\varepsilon}\right)\left(t,\frac{x}{\tilde{\lambda}(t)}\right)e^{-i\frac{\tilde{b(t)}}{4}\frac{|x|^2}{\tilde{\lambda}(t)^2}+i\tilde{\gamma}(t)},\\
\tilde{\lambda}(t)&\in\left(\lambda(t)(1-\overline{\lambda}),\lambda(t)(1+\overline{\lambda})\right),\nonumber\\
\tilde{b}(t)&\in(-\overline{b},\overline{b}),\nonumber\\
\tilde{\gamma}(t)&\in\bigcup_{m\in\mathbb{Z}}(-\overline{\gamma}-\gamma(t)+2m\pi,\overline{\gamma}-\gamma(t)+2m\pi)\nonumber
\end{align}
hold and $\tilde{\varepsilon}$ satisfies the orthogonal conditions
\[
\left(\tilde{\varepsilon},i\Lambda P\right)_2=\left(\tilde{\varepsilon},|y|^2P\right)_2=\left(\tilde{\varepsilon},i\rho\right)_2=0
\]
in $I$. In particular, $\tilde{\lambda}$ and $\tilde{b}$ are unique within functions and $\tilde{\gamma}$ is unique within continuous functions (and is unique within functions under modulo $2\pi$).
\end{lemma}

A summary of the proof is described in Appendix \ref{sec:prfdecom}. See \cite{N} for details of the proof. Also see \cite{LMR,MRUPB}.

\section{Approximate blow-up law}
In this section, we describe the initial values and the approximation functions of the parameters $\lambda$ and $b$ in the decomposition.

\begin{lemma}
Let
\[
\lambda_{\mathrm{app}}(s):=\left(\frac{\alpha}{2}\sqrt{\frac{2\beta}{2-\alpha}}\right)^{-\frac{2}{\alpha}}s^{-\frac{2}{\alpha}},\quad b_{\mathrm{app}}(s):=\frac{2}{\alpha s}.
\]
Then, $(\lambda_{\mathrm{app}},b_{\mathrm{app}})$ is a solution of
\[
\frac{\partial b}{\partial s}+b^2-\beta\lambda^\alpha=0,\quad b+\frac{1}{\lambda}\frac{\partial \lambda}{\partial s}=0
\]
in $s>0$.
\end{lemma}

\begin{lemma}
\label{paraini}
Let define $C_0:=\frac{8E_0}{\||y|Q\|_2^2}$ and $0<\lambda_0\ll 1$ such that $\frac{2\beta}{2-\alpha}+C_0{\lambda_0}^{2-\alpha}>0$. For $\lambda\in(0,\lambda_0]$, we set
\[
\mathcal{F}(\lambda):=\int_\lambda^{\lambda_0}\frac{1}{\mu^{\frac{\alpha}{2}+1}\sqrt{\frac{2\beta}{2-\alpha}+C_0\mu^{2-\alpha}}}d\mu.
\]
Then, for any $s_1\gg 1$, there exist $b_1,\lambda_1>0$ such that
\[
\left|\frac{{\lambda_1}^{\frac{\alpha}{2}}}{\lambda_{\mathrm{app}}(s_1)^{\frac{\alpha}{2}}}-1\right|+\left|\frac{b_1}{b_{\mathrm{app}}(s_1)}-1\right|\lesssim {s_1}^{-\frac{1}{2}}+{s_1}^{2-\frac{4}{\alpha}},\quad \mathcal{F}(\lambda_1)=s_1,\quad E(P_{\lambda_1,b_1,\gamma})=E_0.
\]
Moreover,
\[
\left|\mathcal{F}(\lambda)-\frac{2}{\alpha\lambda^{\frac{\alpha}{2}}\sqrt{\frac{2\beta}{2-\alpha}}}\right|\lesssim\lambda^{-\frac{\alpha}{4}}+\lambda^{2-\frac{3}{2}\alpha}
\]
holds.
\end{lemma}

\begin{proof}
The method of choosing $\lambda_1$ and the estimate of $\mathcal{F}$ are the same as in \cite{LMR} and is therefore omitted.

Setting $h(b):={\lambda_1}^2E(P_{\lambda_1,b,\gamma})$, we have
\begin{align*}
h(b)=&\frac{1}{8}\||y|Q\|_2^2\left(b^2-\frac{2\beta}{2-\alpha}{\lambda_1}^\alpha\right)+O({\lambda_1}^\alpha(b^2+{\lambda_1}^\alpha))\\
=&\frac{1}{8}\||y|Q\|_2^2\left(b^2-b_{\mathrm{app}}(s_1)^2-\frac{2\beta}{2-\alpha}\left({\lambda_1}^\alpha-\lambda_{\mathrm{app}}(s_1)^\alpha\right)\right)+O({\lambda_1}^\alpha(b^2+{\lambda_1}^\alpha)).
\end{align*}
Then, since $\lambda_1$ is sufficiently small if $s_1$ is sufficiently large, we have
\begin{align*}
h(0)-{\lambda_1}^2E_0&=-\frac{1}{8}\||y|Q\|_2^2\frac{2\beta}{2-\alpha}{\lambda_1}^\alpha-{\lambda_1}^2E_0+O({\lambda_1}^{2\alpha})<0,\\
h(1)-{\lambda_1}^2E_0&=\frac{1}{8}\||y|Q\|_2^2\left(1-\frac{2\beta}{2-\alpha}{\lambda_1}^\alpha-{\lambda_1}^2C_0\right)+O({\lambda_1}^\alpha(1+{\lambda_1}^\alpha))>0.
\end{align*}
Therefore, there exists $b_1\in(0,1)$ such that $h(b_1)={\lambda_1}^2E_0$ and we have
\begin{align*}
\left|{b_1}^2-b_{\mathrm{app}}(s_1)^2\right|&\lesssim{\lambda_1}^2+\left|{\lambda_1}^\alpha-\lambda_{\mathrm{app}}(s_1)^\alpha\right|+{\lambda_1}^\alpha\left(\left|{b_1}^2-b_{\mathrm{app}}(s_1)^2\right|+\lambda_{\mathrm{app}}(s_1)^\alpha+{\lambda_1}^\alpha\right)\\
&\lesssim{s_1}^{-\frac{4}{\alpha}}+{s_1}^{-\frac{5}{2}}.
\end{align*}
Consequently, we have the conclusion.
\end{proof}

\section{Uniformity estimates for decomposition}
\label{sec:uniesti}
In this section, we estimate \textit{modulation terms}.

Let define
\[
\mathcal{C}:=\frac{\alpha}{4-\alpha}\left(\frac{\alpha}{2}\sqrt{\frac{2\beta}{2-\alpha}}\right)^{-\frac{4}{\alpha}}.
\]
For $t_1<0$ which is sufficiently close to $0$, we define
\[
s_1:=|\mathcal{C}^{-1}t_1|^{-\frac{\alpha}{4-\alpha}}.
\]
Additionally, let $\lambda_1$ and $b_1$ be given in Lemma \ref{paraini} for $s_1$ and $\gamma_1=0$. Let $u$ be the solution for (NLS$+$) with an initial value
\begin{align}
\label{initial}
u(t_1,x):=P_{\lambda_1,b_1,0}(x).
\end{align}
Then, since $u$ satisfies the assumption of Lemma \ref{decomposition} in a neighbourhood of $t_1$, there exists a decomposition $(\tilde{\lambda}_{t_1},\tilde{b}_{t_1},\tilde{\gamma}_{t_1},\tilde{\varepsilon}_{t_1})$ such that $(\ref{mod})$ in a neighbourhood $I$ of $t_1$. The rescaled time $s_{t_1}$ is defined as
\[
s_{t_1}(t):=s_1-\int_t^{t_1}\frac{1}{\tilde{\lambda}_{t_1}(\tau)^2}d\tau.
\]
Then, we define an inverse function ${s_{t_1}}^{-1}:s_{t_1}(I)\rightarrow I$. Moreover, we define
\begin{align*}
&t_{t_1}:={s_{t_1}}^{-1},\quad \lambda_{t_1}(s):=\tilde{\lambda}(t_{t_1}(s)),\quad b_{t_1}(s):=\tilde{b}(t_{t_1}(s)),\\
&\gamma_{t_1}(s):=\tilde{\gamma}(t_{t_1}(s)),\quad \varepsilon_{t_1}(s,y):=\tilde{\varepsilon}(t_{t_1}(s),y).
\end{align*}
If there is no risk of confusion, the subscript $t_1$ is omitted. In particular, it should be noted that $u\in C((-T_*,T^*),\Sigma^2(\mathbb{R}^N))$ and $|x|\nabla u\in C((-T_*,T^*),L^2(\mathbb{R}^N))$. Furthermore, let $I_{t_1}$ be the maximal interval such that a decomposition as $(\ref{mod})$ is obtained and we define $J_{s_1}:=s\left(I_{t_1}\right)$. Additionally, let $s_0\ (\leq s_1)$ be sufficiently large and let $s':=\max\left\{s_0,\inf J_{s_1}\right\}$.

Let $0<M<\min\{\frac{1}{2},\frac{4}{\alpha}-2\}$ and $s_*$ be defined as
\[
s_*:=\inf\left\{\sigma\in(s',s_1]\ \middle|\ \mbox{(\ref{bootstrap}) holds on }[\sigma,s_1]\right\},
\]
where
\begin{align}
\label{bootstrap}
&\left\|\varepsilon(s)\right\|_{H^1}^2+b(s)^2\||y|\varepsilon(s)\|_2^2<s^{-2K},\quad \left|\frac{\lambda(s)^{\frac{\alpha}{2}}}{\lambda_{\mathrm{app}}(s)^{\frac{\alpha}{2}}}-1\right|+\left|\frac{b(s)}{b_{\mathrm{app}}(s)}-1\right|<s^{-M}.
\end{align}

Finally, we define
\[
\Mod:=\left(\frac{1}{\lambda}\frac{\partial \lambda}{\partial s}+b,\frac{\partial b}{\partial s}+b^2-\theta,1-\frac{\partial \gamma}{\partial s}\right).
\]

In the following discussion, the constant $\epsilon>0$ is a sufficiently small constant. If necessary, $s_0$ and $s_1$ are recalculated in response to $\epsilon>0$.

\begin{lemma}[The equation for $\varepsilon$]
In $J_{s_1}$, 
\begin{align}
\label{epsieq}
&i\frac{\partial \varepsilon}{\partial s}+\Delta \varepsilon-\varepsilon+f\left(P+\varepsilon\right)-f\left(P\right)-\lambda^\alpha \frac{1}{|y|^{2\sigma}}\varepsilon\nonumber\\
&\hspace{20pt}-i\left(\frac{1}{\lambda}\frac{\partial \lambda}{\partial s}+b\right)\Lambda (P+\varepsilon)+\left(1-\frac{\partial \gamma}{\partial s}\right)(P+\varepsilon)+\left(\frac{\partial b}{\partial s}+b^2-\theta\right)\frac{|y|^2}{4}(P+\varepsilon)-\left(\frac{1}{\lambda}\frac{\partial \lambda}{\partial s}+b\right)b\frac{|y|^2}{2}(P+\varepsilon)\nonumber\\
&\hspace{40pt}=-\Psi
\end{align}
holds.
\end{lemma}

\begin{lemma}
For $s\in(s_*,s_1]$,
\[
|(\varepsilon(s),Q)|\lesssim s^{-(K+2)},\quad |\Mod(s)|\lesssim s^{-(K+2)},\quad \|e^{\epsilon'|y|}\Psi\|_{H^1}\lesssim s^{-(K+4)}
\]
hold.
\end{lemma}

\begin{proof}
Let
\[
s_{**}:=\inf\left\{\ s\in[s_*,s_1]\ \middle|\ \left|(\varepsilon(\tau),P)_2\right|<\tau^{-(K+2)}\ \mbox{holds on}\ [s,s_1].\ \right\}.
\]
We work below on the interval $[s_{**},s_1]$.

According to the orthogonality properties, we have
\begin{align}
\label{dortho1}
0&=\frac{d}{ds}\left(i\varepsilon,\Lambda P\right)_2=\left(i\frac{\partial \varepsilon}{\partial s},\Lambda P\right)_2+\left(i\varepsilon,\frac{\partial (\Lambda P)}{\partial s}\right)_2\\
\label{dortho2}
&=\frac{d}{ds}\left(i\varepsilon,i|y|^2 P\right)_2=\left(i\frac{\partial \varepsilon}{\partial s},i|y|^2 P\right)_2+\left(i\varepsilon,i|y|^2 \frac{\partial P}{\partial s}\right)_2\\
\label{dortho3}
&=\frac{d}{ds}\left(i\varepsilon,\rho\right)_2=\left(i\frac{\partial \varepsilon}{\partial s},\rho\right)_2.
\end{align}

For (\ref{dortho1}), we have
\begin{align}
\label{dortho1-1}
\left(i\varepsilon,\frac{\partial (\Lambda P)}{\partial s}\right)_2=\left(i\varepsilon,\frac{\partial}{\partial s}(\lambda^\alpha \Lambda Z)\right)_2=O(s^{-(K+3)})+O(s^{-1}|\Mod(s)|)
\end{align}
and
\begin{align*}
\left(i\frac{\partial \varepsilon}{\partial s},\Lambda P\right)_2&=\left(L_+\re\varepsilon+iL_-\im\varepsilon-\left(f\left(P+\varepsilon\right)-f\left(P\right)-df(Q)(\varepsilon)\right)+\lambda^\alpha \frac{1}{|y|^{2\sigma}}\varepsilon\right.\\
&\hspace{40pt}\left.+i\left(\frac{1}{\lambda}\frac{\partial \lambda}{\partial s}+b\right)\Lambda (P+\varepsilon)-\left(1-\frac{\partial \gamma}{\partial s}\right)(P+\varepsilon)-\left(\frac{\partial b}{\partial s}+b^2-\theta\right)\frac{|y|^2}{4}(P+\varepsilon)\right.\\
&\hspace{80pt}\left.+\left(\frac{1}{\lambda}\frac{\partial \lambda}{\partial s}+b\right)b\frac{|y|^2}{2}(P+\varepsilon)+\Psi,\Lambda P\right)_2.
\end{align*}
According to $\Lambda P_{j,k}^\pm\in H^1(\mathbb{R}^N)\cap C(\mathbb{R}^N)$,
\begin{align*}
\left|\left(L_+\re\varepsilon,\Lambda P\right)_2\right|+\left|\left(iL_-\im\varepsilon,\Lambda P\right)_2\right|+\left|\left(\lambda^\alpha \frac{1}{|y|^{2\sigma}}\varepsilon,\Lambda P\right)_2\right|&=O(s^{-(K+2)}),\\
\left(i\Lambda P,\Lambda P\right)_2=\left(P,\Lambda P\right)_2&=0,\\
\left(\Psi,\Lambda P\right)_2&=O(s^{-2(K+2)})+O(s^{-1}|\Mod(s)|),\\
\left(|y|^2P,\Lambda P\right)_2&=-\||y|Q\|_2^2+O(s^{-2})
\end{align*}
hold. Here, we have
\[
f\left(P+\varepsilon\right)-f\left(P\right)-df(Q)(\varepsilon)=f\left(P+\varepsilon\right)-f\left(P\right)-df(P)(\varepsilon)+df(P)(\varepsilon)-df(Q)(\varepsilon).
\]
Firstly, we consider $\left(f(P+\varepsilon)-f(P)-df(P)(\varepsilon)\right)\Lambda\overline{P}$. For $N\leq 3$, according to Taylor's theorem, we have
\begin{align*}
\left|\left(f(P+\varepsilon)-f(P)-df(P)(\varepsilon)\right)\Lambda\overline{P}\right|\lesssim&(1+|y|^\kappa)(P+|\varepsilon|)^{\frac{4}{N}-1}|\varepsilon|^2Q\\
\lesssim&(1+|y|^\kappa)(Q+|\varepsilon|)^{\frac{4}{N}-1}|\varepsilon|^2Q.
\end{align*}
On the other hand, we assume $N\geq 4$. If $Q<3|\lambda^\alpha Z|$, then $1\lesssim \lambda^\alpha(1+|y|^\kappa)$. Therefore, we have
\[
\left|\left(f(P+\varepsilon)-f(P)-df(P)(\varepsilon)\right)\Lambda\overline{P}\right|\lesssim \lambda^{\alpha}(1+|y|^\kappa)(Q^{\frac{4}{N}}+|\varepsilon|^{\frac{4}{N}})|\varepsilon|Q.
\]
If $3|\lambda^\alpha Z|\leq Q$ and $Q<3|\varepsilon|$, then we have
\[
\left|\left(f(P+\varepsilon)-f(P)-df(P)(\varepsilon)\right)\Lambda\overline{P}\right|\lesssim (1+|y|^\kappa)Q^{\frac{4}{N}}|\varepsilon|^2.
\]
If $3|\varepsilon|\leq Q$, then $P-|\varepsilon|>\frac{1}{3}Q>0$. According to Taylor's theorem, we have
\begin{align*}
\left|\left(f(P+\varepsilon)-f(P)-df(P)(\varepsilon)\right)\Lambda\overline{P}\right|\lesssim&(1+|y|^\kappa)(P-|\varepsilon|)^{\frac{4}{N}-1}|\varepsilon|^2Q\\
\lesssim&(1+|y|^\kappa)Q^{\frac{4}{N}}|\varepsilon|^2.
\end{align*}
Therefore, we have
\[
\left(f(P+\varepsilon)-f(P)-df(P)(\varepsilon),\Lambda P\right)_2=O(s^{-(K+2)}).
\]
The same calculation for $\left(df(P)(\varepsilon)-df(Q)(\varepsilon)\right)\Lambda\overline{P}$ yields
\[
\left(df(P)(\varepsilon)-df(Q)(\varepsilon),\Lambda P\right)_2=O(s^{-(K+2)}).
\]
Accordingly, we have
\[
\left(i\frac{\partial \varepsilon}{\partial s},\Lambda P\right)_2=-\frac{1}{4}\||y|Q\|\left(\frac{\partial b}{\partial s}+b^2-\theta\right)+O(s^{-(K+2)})+O(s^{-1}|\Mod(s)|)
\]
and by (\ref{dortho1}) and (\ref{dortho1-1}),
\[
\frac{\partial b}{\partial s}+b^2-\theta=O(s^{-(K+2)})+O(s^{-1}|\Mod(s)|).
\]

The same calculations for (\ref{dortho2}) and (\ref{dortho3}) yield
\[
\frac{1}{\lambda}\frac{\partial \lambda}{\partial s}+b=O(s^{-(K+2)})+O(s^{-1}|\Mod(s)|),\quad 1-\frac{\partial \gamma}{\partial s}=O(s^{-(K+2)})+O(s^{-1}|\Mod(s)|).
\]
Consequently, we have
\[
|\Mod(s)|\lesssim s^{-(K+2)},\quad\|e^{\epsilon'|y|}\Psi\|_{H^1}\lesssim s^{-(K+4)}.
\]

Finally, since
\[
\|P(s_1)\|_2^2=\|P(s)\|_2^2+2(\varepsilon(s),P(s))_2+\|\varepsilon(s)\|_2^2,
\]
we have
\begin{align*}
\left|(\varepsilon(s),P(s))_2\right|&\lesssim\|\varepsilon(s)\|_2^2+\int_s^{s_1}\left|\left.\frac{d}{ds}\right|_{s=\tau}\|P(s)\|_2^2\right|d\tau\\
&\lesssim s^{-2K}+\int_s^{s_1}\left(\tau^{-2}|\Mod(\tau)|+\tau^{-2(K+2)}\right)d\tau\\
&\lesssim s^{-(K+3)}.
\end{align*}
Therefore, if $s_0$ is sufficiently large, then we have $s_{**}=s_*$. Moreover, we have
\[
\left|(\varepsilon(s),Q)_2\right|\lesssim \left|(\varepsilon(s),P(s))_2\right|+\lambda^\alpha\left|(\varepsilon(s),Z)_2\right|\lesssim s^{-(K+2)}.
\]
\end{proof}

\section{Modified energy function}
\label{sec:MEF}
In this section, we proceed with a modified version of the technique presented in Le Coz, Martel, and Rapha\"{e}l \cite{LMR} and Martel and Szeftel \cite{RSEU}. Let $m>0$ be sufficiently large and define
\begin{align*}
H(s,\varepsilon)&:=\frac{1}{2}\left\|\varepsilon\right\|_{H^1}^2+b(s)^2\left\||y|\varepsilon\right\|_2^2-\int_{\mathbb{R}^N}\left(F(P+\varepsilon)-F(P)-dF(P)(\varepsilon)\right)dy-\frac{1}{2}\lambda^\alpha\left\||y|^{-\sigma}\varepsilon\right\|_2^2,\\
S(s,\varepsilon)&:=\frac{1}{\lambda^m}H(s,\varepsilon).
\end{align*}

\begin{lemma}[Coercivity of $H$]
\label{Hcoer}
For $s\in(s_*,s_1]$, 
\[
\|\varepsilon\|_{H^1}^2+b^2\left\||y|\varepsilon\right\|_2^2+O(s^{-2(K+2)})\lesssim H(s,\varepsilon)\lesssim \|\varepsilon\|_{H^1}^2+b^2\left\||y|\varepsilon\right\|_2^2
\]
hold.
\end{lemma}

\begin{proof}
If $N\leq 3$, then we have
\[
\left|F(P+\varepsilon)-F(P)-dF(P)(\varepsilon)-\frac{1}{2}d^2F(P)(\varepsilon,\varepsilon)\right|\lesssim \left(|P|^{\frac{4}{N}-1}+|\varepsilon|^{\frac{4}{N}-1}\right)|\varepsilon|^3.
\]
For $N\geq 4$, if $2|\varepsilon|\geq |P|$, then we have
\[
\left|F(P+\varepsilon)-F(P)-dF(P)(\varepsilon)-\frac{1}{2}d^2F(P)(\varepsilon,\varepsilon)\right|\lesssim |\varepsilon|^{\frac{4}{N}+2}.
\]
If $2|\varepsilon|<|P|$, then $|P|>0$ and $|P|-|\varepsilon|>\frac{1}{2}|P|$. Therefore, we have
\[
\left|F(P+\varepsilon)-F(P)-dF(P)(\varepsilon)-\frac{1}{2}d^2F(P)(\varepsilon,\varepsilon)\right|\lesssim \left(|P|-|\varepsilon|\right)^{\frac{4}{N}-1}|\varepsilon|^3\lesssim |\varepsilon|^{\frac{4}{N}+2}.
\]
Therefore, we obtain
\[
\int_{\mathbb{R}^N}\left(F(P+\varepsilon)-F(P)-dF(P)(\varepsilon)-\frac{1}{2}d^2F(P)(\varepsilon,\varepsilon)\right)dy=o(\|\varepsilon\|_{H^1}^2).
\]

Similarly, if $N\leq 3$, then we have
\[
\left|\frac{1}{2}d^2F(P)(\varepsilon,\varepsilon)-\frac{1}{2}d^2F(Q)(\varepsilon,\varepsilon)\right|\lesssim \lambda^\alpha\left(Q^{\frac{4}{N}-1}+|\lambda^\alpha Z|^{\frac{4}{N}-1}\right)|\varepsilon|^2|Z|.
\]
For $N\geq 4$, if $2|\lambda^\alpha Z|\geq Q$, then we have
\[
\left|\frac{1}{2}d^2F(P)(\varepsilon,\varepsilon)-\frac{1}{2}d^2F(Q)(\varepsilon,\varepsilon)\right|\lesssim |\lambda^\alpha Z|^{\frac{4}{N}}|\varepsilon|^2.
\]
If $2|\lambda^\alpha Z|<Q$, then $Q-|\lambda^\alpha Z|>\frac{1}{2}Q$. Therefore, we have
\[
\left|\frac{1}{2}d^2F(P)(\varepsilon,\varepsilon)-\frac{1}{2}d^2F(Q)(\varepsilon,\varepsilon)\right|\lesssim \lambda^\alpha\left(Q-|\lambda^\alpha Z|\right)^{\frac{4}{N}-1}|\varepsilon|^2|Z|\lesssim (1+|\cdot|^\kappa)\lambda^\alpha|\varepsilon|^2Q^\frac{4}{N}
\]
and
\[
\int_{\mathbb{R}^N}\left(\frac{1}{2}d^2F(P)(\varepsilon,\varepsilon)-\frac{1}{2}d^2F(Q)(\varepsilon,\varepsilon)\right)dy=o(\|\varepsilon\|_{H^1}^2).
\]

Accordingly, we have
\begin{align*}
\left\|\varepsilon\right\|_{H^1}^2-\int_{\mathbb{R}^N}d^2F(Q)(\varepsilon,\varepsilon)dy&=\left\langle L_+\re\varepsilon,\re\varepsilon\right\rangle+\left\langle L_-\im\varepsilon,\im\varepsilon\right\rangle\\
&\geq\mu\|\varepsilon\|_{H^1}^2-\frac{1}{\mu}\left((\re\varepsilon,Q)_2^2+(\re\varepsilon,|y|^2 Q)_2^2+(\im\varepsilon,\rho)_2^2\right)\\
&=\mu\|\varepsilon\|_{H^1}^2-\frac{1}{\mu}\left((\varepsilon,Q)_2^2+\left((\varepsilon,|y|^2P)_2-\lambda^\alpha(\varepsilon,|y|^2Z)_2\right)^2+(\varepsilon,i\rho)_2^2\right)\\
&=\mu\|\varepsilon\|_{H^1}^2+O(s^{-2(K+2)}).
\end{align*}
Consequently, we have the lower estimate. The upper estimate is clearly.
\end{proof}

\begin{corollary}[Estimation of $S$]
\label{Sesti}
For $s\in(s_*,s_1]$, 
\[
\frac{1}{\lambda^m}\left(\|\varepsilon\|_{H^1}^2+b^2\left\||y|\varepsilon\right\|_2^2+O(s^{-2(K+2)})\right)\lesssim S(s,\varepsilon)\lesssim \frac{1}{\lambda^m}\left(\|\varepsilon\|_{H^1}^2+b^2\left\||y|\varepsilon\right\|_2^2\right)
\]
hold.
\end{corollary}

\begin{lemma}
\label{Lambda}
For $s\in(s_*,s_1]$, 
\begin{align}
\label{flambdaesti}
\left|\left(f(P+\varepsilon)-f(P),\Lambda \varepsilon\right)_2\right|&\lesssim \|\varepsilon\|_{H^1}^2+s^{-3K}
\end{align}
holds.
\end{lemma}

\begin{proof}
Calculated in the same way as in Section 5.4 in \cite{LMR}, we have
\begin{align*}
&\nabla\left(F(P+\varepsilon)-F(P)-dF(P)(\varepsilon)\right)\\
=&\re\left(f(P+\varepsilon)\nabla\left(\overline{P}+\overline{\varepsilon}\right)-f(P)\nabla\overline{P}-df(P)(\varepsilon)\nabla\overline{P}-f(P)\nabla\overline{\varepsilon}\right)\\
=&\re\left(\left(f(P+\varepsilon)-f(P)-df(P)(\varepsilon)\right)\nabla\overline{P}+\left(f(P+\varepsilon)-f(P)\right)\nabla\overline{\varepsilon}\right)
\end{align*}
and
\begin{align*}
&\left(f(P+\varepsilon)-f(P),\Lambda \varepsilon\right)=\re\int_{\mathbb{R}^N}\left(f(P+\varepsilon)-f(P)\right)\Lambda\overline{\varepsilon}dy\\
=&\re\int_{\mathbb{R}^N}\left(\frac{N}{2}\left(f(P+\varepsilon)-f(P)\right)\overline{\varepsilon}-\left(f(P+\varepsilon)-f(P)-df(P)(\varepsilon)\right)y\cdot\nabla\overline{P}-N\left(F(P+\varepsilon)-F(P)-dF(P)(\varepsilon)\right)\right)dy.
\end{align*}

Firstly,
\begin{align*}
\left|\left(f(P+\varepsilon)-f(P)\right)\overline{\varepsilon}\right|+\left|F(P+\varepsilon)-F(P)-dF(P)(\varepsilon)\right|\lesssim&((1+|y|^\kappa)Q^{\frac{4}{N}}+|\varepsilon|^{\frac{4}{N}})|\varepsilon|^2
\end{align*}
holds.

Next, we consider $\left(f(P+\varepsilon)-f(P)-df(P)(\varepsilon)\right)y\cdot\nabla\overline{P}$. For $N\leq 3$, we have
\[
\left|\left(f(P+\varepsilon)-f(P)-df(P)(\varepsilon)\right)y\cdot\nabla\overline{P}\right|\lesssim(1+|y|^\kappa)(Q+|\varepsilon|)^{\frac{4}{N}-1}|\varepsilon|^2Q.
\]
For $N\geq 4$, if $Q<3|\lambda^\alpha Z|$, then $1\lesssim \lambda^\alpha(1+|y|^\kappa)$. Therefore, we have
\[
\left|\left(f(P+\varepsilon)-f(P)-df(P)(\varepsilon)\right)y\cdot\nabla\overline{P}\right|\lesssim \lambda^{K\alpha}(1+|y|^\kappa)(Q^{\frac{4}{N}}+|\varepsilon|^{\frac{4}{N}})|\varepsilon|Q.
\]
If $3|\lambda^\alpha Z|\leq Q$ and $Q<3|\varepsilon|$, we have
\[
\left|\left(f(P+\varepsilon)-f(P)-df(P)(\varepsilon)\right)y\cdot\nabla\overline{P}\right|\lesssim (1+|y|^\kappa)Q^{\frac{4}{N}}|\varepsilon|^2.
\]
If $3|\varepsilon|\leq Q$, then $P-|\varepsilon|\geq\frac{1}{3}Q>0$. Therefore, we have
\[
\left|\left(f(P+\varepsilon)-f(P)-df(P)(\varepsilon)\right)y\cdot\nabla\overline{P}\right|\lesssim(1+|y|^\kappa)Q^{1-\frac{4}{N}}|\varepsilon|^2.
\]
Consequently, we have the conclusion.
\end{proof}

\begin{lemma}[Derivative of $H$ in time]
\label{Hdiff}
For $s\in(s_*,s_1]$, 
\[
\frac{d}{ds}H(s,\varepsilon(s))\gtrsim -b\left(\|\varepsilon\|_{H^1}^2+b^2\left\||y|\varepsilon\right\|_2^2\right)+O(s^{-2(K+2)})
\]
holds.
\end{lemma}

\begin{proof}
Firstly, we have
\[
\frac{d}{ds}H(s,\varepsilon(s))=\frac{\partial H}{\partial s}(s,\varepsilon(s))+\left(i\frac{\partial H}{\partial \varepsilon}(s,\varepsilon(s)),i\frac{\partial \varepsilon}{\partial s}(s)\right)_2.
\]
Here,
\begin{align*}
\frac{\partial H}{\partial \varepsilon}=&-\Delta \varepsilon+\varepsilon+2b^2|y|^2\varepsilon-\left(f(P+\varepsilon)-f(P)\right)-\frac{\lambda^\alpha}{|y|^{2\sigma}}\varepsilon\\
=&L_+\re\varepsilon+iL_-\im\varepsilon+2b^2|y|^2\varepsilon-\left(f(P+\varepsilon)-f(P)-df(Q)(\varepsilon)\right)-\frac{\lambda^\alpha}{|y|^{2\sigma}}\varepsilon,\\
\frac{\partial H}{\partial s}=&2b\frac{\partial b}{\partial s}\||y|\varepsilon\|_2^2-\re\int_{\mathbb{R}^N}\left(f(P+\varepsilon)-f(P)-df(P)(\varepsilon)\right)\frac{\partial \overline{P}}{\partial s}dy-\frac{\alpha \lambda^\alpha}{2}\frac{1}{\lambda}\frac{\partial \lambda}{\partial s}\||y|^{-\sigma}\varepsilon\|_2^2
\end{align*}
hold. Therefore, we have
\[
\frac{\partial H}{\partial s}\gtrsim -b^3\||y|\varepsilon\|_2^2-s^{-2}b\|\varepsilon\|_{H^1}^2+O(s^{-3K}).
\]
Let define
\[
\Modop v:=i\left(\frac{1}{\lambda}\frac{\partial \lambda}{\partial s}+b\right)\Lambda v-\left(1-\frac{\partial \gamma}{\partial s}\right)v-\left(\frac{\partial b}{\partial s}+b^2-\theta\right)\frac{|y|^2}{4}v+\left(\frac{1}{\lambda}\frac{\partial \lambda}{\partial s}+b\right)b\frac{|y|^2}{2}v.
\]
Then,
\[
i\frac{\partial \varepsilon}{\partial s}=\frac{\partial H}{\partial \varepsilon}-2b^2|y|^2\varepsilon+\Modop (P+\varepsilon)+\Psi
\]
holds. Moreover, we have
\[
\left(i\frac{\partial H}{\partial \varepsilon}(s,\varepsilon(s)),i\frac{\partial \varepsilon}{\partial s}(s)\right)_2=\left(i\frac{\partial H}{\partial \varepsilon}(s,\varepsilon(s)),-2b^2|y|^2\varepsilon+\Modop (P+\varepsilon)+\Psi\right)_2.
\]

Secondly, we have
\begin{align*}
\left(i\frac{\partial H}{\partial \varepsilon}(s,\varepsilon),-2b^2|y|^2\varepsilon\right)_2=&-4b^2\left(i\nabla \varepsilon,y\varepsilon\right)_2+\left(i\left(|P+\varepsilon|^{\frac{4}{N}}-|P|^{\frac{4}{N}}\right)P,-2b^2|y|^2\varepsilon\right)_2\\
=&-4b^2\left(i\nabla \varepsilon,y\varepsilon\right)_2+O(b^2\|\varepsilon\|_{H^1}^2+s^{-3K})\\
\gtrsim&-b\left(\|\nabla\varepsilon\|_2^2+b^2\||y|^2\varepsilon\|_2^2\right)+O(b^2\|\varepsilon\|_{H^1}^2+s^{-3K}).
\end{align*}

Thirdly,
\begin{align*}
\left(i\frac{\partial H}{\partial \varepsilon}(s,\varepsilon(s)),\Modop P\right)_2=O(s^{-(3K+2)}),\quad \left(i\frac{\partial H}{\partial \varepsilon}(s,\varepsilon(s)),\Psi\right)_2=O(s^{-2(K+2)})
\end{align*}
hold.

Finally, since
\[
\left|\left(i\left(f(P+\varepsilon)-f(P)\right),i\Lambda \varepsilon\right)_2\right|+\left|\left(i\left(f(P+\varepsilon)-f(P)\right),|y|^2\varepsilon\right)_2\right|=O(\|\varepsilon\|_{H^1}^2)+O(s^{-3K}),
\]
we have
\[
\left(i\frac{\partial H}{\partial \varepsilon}(s,\varepsilon(s)),\Modop \varepsilon\right)_2=o\left(b\left(\|\varepsilon\|_{H^1}^2+b^2\left\||y|\varepsilon\right\|_2^2\right)\right)+O(s^{-(5K+2)}).
\]

Consequently, we have the conclusion.
\end{proof}

\begin{lemma}[Derivative of $S$ in time]
\label{Sdiff}
Let $m>0$ be sufficiently large. Then,
\[
\frac{d}{ds}S(s,\varepsilon(s))\gtrsim \frac{b}{\lambda^m}\left(\|\varepsilon\|_{H^1}^2+b^2\left\||y|\varepsilon\right\|_2^2+O(s^{-(2K+3)})\right)
\]
holds for $s\in(s_*,s_1]$.
\end{lemma}

\begin{proof}
From Lemma \ref{Hdiff}, we have
\begin{align*}
\frac{d}{ds}S(s,\varepsilon(s))=&-m\frac{1}{\lambda}\frac{\partial\lambda}{\partial s}\frac{1}{\lambda^m}H(s,\varepsilon)+\frac{1}{\lambda^m}\frac{d}{ds}H(s,\varepsilon(s))\\
=&-m\left(\frac{1}{\lambda}\frac{\partial\lambda}{\partial s}+b\right)\frac{1}{\lambda^m}H(s,\varepsilon)+m\frac{b}{\lambda^m}H(s,\varepsilon)+\frac{1}{\lambda^m}\frac{d}{ds}H(s,\varepsilon(s))\\
\geq&\frac{b}{\lambda^m}\left((m-\epsilon)C\left(\|\varepsilon\|_{H^1}^2+b^2\left\||y|\varepsilon\right\|_2^2\right)+O(s^{-2(K+2)})-C'\left(\|\varepsilon\|_{H^1}^2+b^2\left\||y|\varepsilon\right\|_2^2\right)+O(s^{-(2K+3)})\right).
\end{align*}
Therefore, we have the conclusion if $m$ is sufficiently large.
\end{proof}

\section{Bootstrap}
\label{sec:bootstrap}
In this section, we use the estimates obtained in Section \ref{sec:MEF} and the bootstrap to establish the estimates of the parameters.

\begin{lemma}[Re-estimation]
\label{rebootstrap}
For $s\in(s_*,s_1]$, 
\begin{align}
\label{reepsiesti}
\left\|\varepsilon(s)\right\|_{H^1}^2+b(s)^2\left\||y|\varepsilon(s)\right\|_2^2&\lesssim s^{-(2K+2)},\\
\label{reesti}
\left|\frac{\lambda(s)^{\frac{\alpha}{2}}}{\lambda_{\mathrm{app}}(s)^{\frac{\alpha}{2}}}-1\right|+\left|\frac{b(s)}{b_{\mathrm{app}}(s)}-1\right|&\lesssim s^{-\frac{1}{2}}+s^{2-\frac{4}{\alpha}}
\end{align}
holds.
\end{lemma}

\begin{proof}
We prove $(\ref{reepsiesti})$ by contradiction. Let $C_\dagger>0$ be sufficiently large and define
\[
s_\dagger:=\inf\left\{\sigma\in(s_*,s_1]\ \middle|\ \left\|\varepsilon(\tau)\right\|_{H^1}^2+b(\tau)^2\left\||y|\varepsilon(\tau)\right\|_2^2\leq C_\dagger \tau^{-2(K+1)}\ (\tau\in[\sigma,s_1])\right\}.
\]
Then, $s_\dagger<s_1$ holds. Here, we assume that $s_\dagger>s_*$. Then, we have
\[
\left\|\varepsilon(s_\dagger)\right\|_{H^1}^2+b(s_\dagger)^2\left\||y|\varepsilon(s_\dagger)\right\|_2^2=C_\dagger {s_\dagger}^{-2(K+1)}.
\]
Let $C_\ddagger>\epsilon$ and define
\[
s_\ddagger:=\sup\left\{\sigma\in(s_*,s_1]\ \big|\ \left\|\varepsilon(\tau)\right\|_{H^1}^2+b(\tau)^2\left\||y|\varepsilon(\tau)\right\|_2^2\geq C_\ddagger\tau^{-2(K+1)}\ (\tau\in[s_\dagger,\sigma])\right\}.
\]
Then, we have $s_\ddagger>s_\dagger$. Furthermore,
\[
\left\|\varepsilon(s_\ddagger)\right\|_{H^1}^2+b(s_\ddagger)^2\left\||y|\varepsilon(s_\ddagger)\right\|_2^2=C_\ddagger {s_\ddagger}^{-2(K+1)}.
\]
Then, according to Corollary \ref{Sesti} and Lemma \ref{Sdiff}, we have
\begin{align*}
\frac{C_1}{\lambda^m}\left(\left\|\varepsilon\right\|_{H^1}^2+b^2\left\||y|\varepsilon\right\|_2^2-C's^{-2(K+1)}\right)&\leq S(s,\varepsilon)\leq\frac{C_2}{\lambda^m}\left(\left\|\varepsilon\right\|_{H^1}^2+b^2\left\||y|\varepsilon\right\|_2^2\right),\\
\frac{b}{\lambda^m}\left(\left\|\varepsilon\right\|_{H^1}^2+b^2\left\||y|\varepsilon\right\|_2^2-\epsilon s^{-2(K+1)}\right)&\lesssim \frac{d}{ds}S(s,\varepsilon).
\end{align*}
in $(s_*,s_1]$. Therefore, we have
\begin{align*}
C_1(C_\dagger-C'){s_\dagger}^{-2(K+1)}=&C_1\left(\left\|\varepsilon(s_\dagger)\right\|_{H^1}^2+b(s_\dagger)^2\left\||y|\varepsilon(s_\dagger)\right\|_2^2-C'{s_\dagger}^{-2(K+1)}\right)\\
\leq&\lambda(s_\dagger)^mS(s_\dagger,\varepsilon(s_\dagger))\\
\leq&\lambda(s_\dagger)^mS(s_\ddagger,\varepsilon(s_\ddagger))\\
\leq&C_2\frac{\lambda(s_\dagger)^m}{\lambda(s_\ddagger)^m}\left(\left\|\varepsilon(s_\ddagger)\right\|_{H^1}^2+b(s_\ddagger)^2\left\||y|\varepsilon(s_\ddagger)\right\|_2^2\right)\\
\leq&C_2C_\ddagger\frac{\lambda(s_\dagger)^m}{\lambda(s_\ddagger)^m}{s_\ddagger}^{-2(K+1)}\\
\leq&(1+\epsilon)C_2C_\ddagger \frac{{s_\dagger}^{-\frac{2m}{\alpha}}}{{s_\ddagger}^{-\frac{2m}{\alpha}}}\frac{{s_\ddagger}^{-2(K+1)}}{{s_\dagger}^{-2(K+1)}}{s_\dagger}^{-2(K+1)}
\end{align*}
and since $K-\frac{m}{\alpha}>0$, we have
\[
C_1(C_\dagger-C')\leq (1+\epsilon)C_2C_\ddagger.
\]
Since $C_\dagger$ is sufficiently large, it is a contradiction. Therefore, $s_\dagger\leq s_*$. On the other hand, $s_\dagger\geq s_*$ is clearly. Accordingly, $s_*=s_\dagger$.

Next, since
\[
\left|E(P_{\lambda,b,\gamma}(s))-E_0\right|\leq\left|\int_{s_1}^s\left.\frac{d}{ds}\right|_{s=\tau}E(P_{\lambda,b,\gamma}(s)d\tau\right|\leq\int_s^{s_1}\tau^{-(K+2)+\frac{4}{\alpha}}d\tau\lesssim s^{-(K+1)+\frac{4}{\alpha}},
\]
we have
\[
\left|b^2-\frac{2\beta}{2-\alpha}\lambda^\alpha-C_0\lambda^2\right|\leq\lambda^2\left(\left|\frac{b^2}{\lambda^2}-\frac{2\beta}{2-\alpha}\lambda^{\alpha-2}-\frac{8}{\||y|Q\|_2^2}E(P_{\lambda,b,\gamma})\right|+\frac{8}{\||y|Q\|_2^2}\left|E(P_{\lambda,b,\gamma})-E_0\right|\right)\lesssim s^{-4}.
\]
From the definition of $\mathcal{F}$, we have
\[
\left|\mathcal{F}'(s)-1\right|\lesssim s^{-2}.
\]
Therefore, we have
\[
\left|s-\mathcal{F}(\lambda(s))\right|\lesssim s^{-1}
\]
since $\mathcal{F}(\lambda(s_1))=s_1$. From definition $\lambda_{\mathrm{app}}$, we have
\[
\left|\frac{{\lambda_{\rm{app}}(s)}^\frac{\alpha}{2}}{\lambda(s)^\frac{\alpha}{2}}-1\right|\lesssim s^{-\frac{1}{2}}+s^{2-\frac{4}{\alpha}}
\]
and
\[
\left|\frac{\lambda(s)^\frac{\alpha}{2}}{{\lambda_{\rm{app}}(s)}^\frac{\alpha}{2}}-1\right|\leq \left|\frac{\lambda(s)^\frac{\alpha}{2}}{{\lambda_{\rm{app}}(s)}^\frac{\alpha}{2}}\right|\left|\frac{{\lambda_{\rm{app}}(s)}^\frac{\alpha}{2}}{\lambda(s)^\frac{\alpha}{2}}-1\right|\lesssim s^{-\frac{1}{2}}+s^{2-\frac{4}{\alpha}}.
\]

Finally, we have
\[
\left|b(s)^2-{b_{\mathrm{app}}(s)}^2\right|\lesssim s^{-4}+s^{-2-\frac{1}{2}}+s^{-\frac{4}{\alpha}}
\]
and
\[
\left|\frac{b(s)}{b_{\mathrm{app}}(s)}-1\right|\lesssim s^{-\frac{1}{2}}+s^{2-\frac{4}{\alpha}}.
\]

Consequently, we have the conclusion.
\end{proof}

\begin{corollary}
\label{reesti}
If $s_0$ is sufficiently large, then $s_*=s'$.
\end{corollary}

\begin{lemma}
If $s_0$ is sufficiently large, then $s'=s_0$.
\end{lemma}

\begin{proof}
See \cite{N} for the proof.
\end{proof}

\section{Conversion of estimates}
\label{sec:convesti}
In this section, we rewrite the estimates obtained for the time variable $s$ in Lemma \ref{rebootstrap} into estimates for the time variable $t$.

\begin{lemma}[Interval]
\label{interval}
If $s_0$ is sufficiently large, then there is $t_0<0$ which is sufficiently close to $0$ such that for $t_1\in(t_0,0)$, 
\[
[t_0,t_1]\subset {s_{t_1}}^{-1}([s_0,s_1]),\quad \left|\mathcal{C}s_{t_1}(t)^{-\frac{4-\alpha}{\alpha}}-|t|\right|\lesssim |t|^{1+\frac{\alpha M}{4-\alpha}}\ (t\in [t_0,t_1])
\]
holds.
\end{lemma}

\begin{proof}
Since $t_{t_1}(s_1)=t_1$ and $s_1=|\mathcal{C}^{-1}t_1|^{-\frac{\alpha}{4-\alpha}}$, we have
\begin{align*}
\int_s^{s_1}\lambda_{\mathrm{app}}(\tau)^2\left(\frac{\lambda_{t_1}(\tau)}{\lambda_{\mathrm{app}}(\tau)}-1\right)\left(\frac{\lambda_{t_1}(\tau)}{\lambda_{\mathrm{app}}(\tau)}+1\right)d\tau&=\int_s^{s_1}\left(\lambda_{t_1}(\tau)^2-\lambda_{\mathrm{app}}(\tau)^2\right)d\tau\\
&=t_{t_1}(s_1)-t_{t_1}(s)+\mathcal{C}({s_1}^{1-\frac{4}{\alpha}}-s^{1-\frac{4}{\alpha}})\\
&=|t_{t_1}(s)|-\mathcal{C}s^{-\frac{4-\alpha}{\alpha}}.
\end{align*}
Therefore, we have
\[
\left||t_{t_1}(s)|-\mathcal{C}s^{-\frac{4-\alpha}{\alpha}}\right|\lesssim\int_s^{s_1}\lambda_{\mathrm{app}}(\tau)^2\tau^{-M}d\tau\lesssim\int_s^{s_1}\tau^{-\frac{4}{\alpha}-M}d\tau\leq\frac{\alpha}{M+4-\alpha}s^{-\left(\frac{4-\alpha}{\alpha}+M\right)}.
\]
Accordingly,
\[
|t_{t_1}(s)|\approx s^{-\frac{4-\alpha}{\alpha}}\quad \mbox{i.e.}\ |t|\approx s_{t_1}(t)^{-\frac{4-\alpha}{\alpha}}.
\]
\end{proof}

\begin{lemma}[Conversion of estimates]
\label{uniesti}
Let
\[
\mathcal{C}_\lambda:=\mathcal{C}^{-\frac{2}{4-\alpha}}\left(\frac{\alpha}{2}\sqrt{\frac{2\beta}{2-\alpha}}\right)^{-\frac{2}{\alpha}},\quad \mathcal{C}_b:=\frac{2}{\alpha}\mathcal{C}^{-\frac{\alpha}{4-\alpha}}.
\]
For $t\in[t_0,t_1]$, 
\begin{align*}
&\tilde{\lambda}_{t_1}(t)=\mathcal{C}_\lambda|t|^\frac{2}{4-\alpha}\left(1+\epsilon_{\tilde{\lambda},t_1}(t)\right),\quad \tilde{b}_{t_1}(t)=\mathcal{C}_b|t|^\frac{\alpha}{4-\alpha}\left(1+\epsilon_{\tilde{b},t_1}(t)\right),\\
&\|\tilde{\varepsilon}_{t_1}(t)\|_{H^1}\lesssim |t|^{\frac{\alpha K}{4-\alpha}},\quad \||y|\tilde{\varepsilon}_{t_1}(t)\|_2\lesssim |t|^{\frac{\alpha (K-1)}{4-\alpha}}
\end{align*}
hold. Furthermore,
\[
\sup_{t_1\in[t,0)}\left|\epsilon_{\tilde{\lambda},t_1}(t)\right|\lesssim |t|^\frac{\alpha M}{4-\alpha},\ \sup_{t_1\in[t,0)}\left|\epsilon_{\tilde{b},t_1}(t)\right|\lesssim |t|^\frac{\alpha M}{4-\alpha}.
\]
\end{lemma}

\begin{proof}
Let
\[
\epsilon_{\tilde{\lambda},t_1}(t):=\frac{\tilde{\lambda}_{t_1}(t)}{\mathcal{C}_\lambda|t|^\frac{2}{4-\alpha}}-1.
\]
Then, we have
\[
\left|\epsilon_{\tilde{\lambda},t_1}(t)\right|\leq\left|\frac{\tilde{\lambda}_{t_1}(t)}{\lambda_{\mathrm{app}}(s_{t_1}(t))}-1\right|\left|\frac{\lambda_{\mathrm{app}}(s_{t_1}(t))}{\mathcal{C}_\lambda|t|^\frac{2}{4-\alpha}}\right|+\frac{1}{\mathcal{C}_\lambda|t|^\frac{2}{4-\alpha}}\left|\lambda_{\mathrm{app}}(s_{t_1}(t))-\mathcal{C}_\lambda|t|^\frac{2}{4-\alpha}\right|\lesssim|t|^\frac{\alpha M}{4-\alpha}.
\]

The same is done for $\epsilon_{\tilde{b},t_1}(t)$.
\end{proof}

\section{Proof of Theorem \ref{theorem:EMBS}}
\label{sec:proof}
See \cite{LMR,N} for details of proof.

\begin{proof}[proof of Theorem \ref{theorem:EMBS}]
Let $(t_n)_{n\in\mathbb{N}}\subset(t_0,0)$ be a monotonically increasing sequence such that $\lim_{n\nearrow \infty}t_n=0$. For each $n\in\mathbb{N}$, $u_n$ is the solution for (NLS$+$) with an initial value
\begin{align*}
u_n(t_n,x):=P_{\lambda_{1,n},b_{1,n},0}(x)
\end{align*}
at $t_n$, where $b_{1,n}$ and $\lambda_{1,n}$ are given by Lemma \ref{paraini} for $t_n$.

According to Lemma \ref{decomposition} with an initial value $\tilde{\gamma}_n(t_n)=0$, there exists a decomposition
\[
u_n(t,x)=\frac{1}{\tilde{\lambda}_n(t)^{\frac{N}{2}}}\left(P+\tilde{\varepsilon}_n\right)\left(t,\frac{x}{\tilde{\lambda}_n(t)}\right)e^{-i\frac{\tilde{b}_n(t)}{4}\frac{|x|^2}{\tilde{\lambda}_n(t)^2}+i\tilde{\gamma}_n(t)}.
\]
Then, $(u_n(t_0))_{n\in\mathbb{N}}$ is bounded in $\Sigma^1$. Therefore, up to a subsequence, there exists $u_\infty(t_0)\in \Sigma^1$ such that
\[
u_n(t_0)\rightharpoonup u_\infty(t_0)\quad \mathrm{in}\ \Sigma^1,\quad u_n(t_0)\rightarrow u_\infty(t_0)\quad \mathrm{in}\ L^2(\mathbb{R}^N)\quad (n\rightarrow\infty),
\]
see \cite{LMR,N} for details.

Let $u_\infty$ be the solution for (NLS$+$) with an initial value $u_\infty(t_0)$ and $T^*$ be the supremum of the maximal existence interval of $u_\infty$. Moreover, we define $T:=\min\{0,T^*\}$. Then, for any $T'\in[t_0,T)$, $[t_0,T']\subset[t_0,t_n]$ if $n$ is sufficiently large. Then, there exist $n_0$ and $C(T',t_0)>0$ such that 
\[
\sup_{n\geq n_0}\|u_n\|_{L^\infty([t_0,T'],\Sigma^1)}\leq C(T',t_0)
\]
holds. According to Lemma B.2 in \cite{N}, 
\[
u_n\rightarrow u_\infty\quad \mathrm{in}\ C\left([t_0,T'],L^2(\mathbb{R}^N)\right)\quad (n\rightarrow\infty)
\]
holds. In particular, $u_n(t)\rightharpoonup u_\infty(t)\ \mathrm{in}\ \Sigma^1$ for any $t\in [t_0,T)$. Furthermore, from the mass conservation, we have
\[
\|u_\infty(t)\|_2=\|u_\infty(t_0)\|_2=\lim_{n\rightarrow\infty}\|u_n(t_0)\|_2=\lim_{n\rightarrow\infty}\|u_n(t_n)\|_2=\lim_{n\rightarrow\infty}\|P(t_n)\|_2=\|Q\|_2.
\]

Based on weak convergence in $H^1(\mathbb{R}^N)$ and Lemma \ref{decomposition}, we decompose $u_\infty$ to
\[
u_\infty(t,x)=\frac{1}{\tilde{\lambda}_\infty(t)^{\frac{N}{2}}}\left(P+\tilde{\varepsilon}_\infty\right)\left(t,\frac{x}{\tilde{\lambda}_\infty(t)}\right)e^{-i\frac{\tilde{b}_\infty(t)}{4}\frac{|x|^2}{{\tilde{\lambda}_\infty(t)}^2}+i\tilde{\gamma}_\infty(t)},
\]
where an initial value of $\tilde{\gamma}_\infty$ is $\gamma_\infty(t_0)\in\left(|t_0|^{-1}-\pi,|t_0|^{-1}+\pi\right]\cap\tilde{\gamma}(u_\infty(t_0))$ (which is unique, see \cite{N}). Furthermore, for any $t\in[t_0,T)$, as $n\rightarrow\infty$, 
\[
\tilde{\lambda}_n(t)\rightarrow\tilde{\lambda}_\infty(t),\quad \tilde{b}_n(t)\rightarrow \tilde{b}_\infty(t),\quad e^{i\tilde{\gamma}_n(t)}\rightarrow e^{i\tilde{\gamma}_\infty(t)},\quad\tilde{\varepsilon}_n(t)\rightharpoonup \tilde{\varepsilon}_\infty(t)\quad \mathrm{in}\ \Sigma^1
\]
holds. Consequently, from a uniform estimate of Lemma \ref{uniesti}, as $n\rightarrow\infty$, we have
\begin{align*}
&\tilde{\lambda}_{\infty}(t)=\mathcal{C}_\lambda\left|t\right|^\frac{2}{4-\alpha}(1+\epsilon_{\tilde{\lambda},0}(t)),\quad \tilde{b}_{\infty}(t)=\mathcal{C}_b\left|t\right|^\frac{\alpha}{4-\alpha}(1+\epsilon_{\tilde{b},0}(t)),\\
&\|\tilde{\varepsilon}_{\infty}(t)\|_{H^1}\lesssim \left|t\right|^{\frac{\alpha K}{4-\alpha}},\quad \||y|\tilde{\varepsilon}_{\infty}(t)\|_2\lesssim \left|t\right|^{\frac{\alpha (K-1)}{4-\alpha}},\quad\left|\epsilon_{\tilde{\lambda},0}(t)\right|\lesssim |t|^\frac{\alpha M}{4-\alpha},\quad \left|\epsilon_{\tilde{b},0}(t)\right|\lesssim |t|^{\frac{\alpha M}{4-\alpha}}.
\end{align*}
Consequently, we obtain that $u$ converges to the blow-up profile in $\Sigma^1$.

Finally, we check energy of $u_\infty$. Since
\[
E\left(u_n\right)-E\left(P_{\tilde{\lambda}_n,\tilde{b}_n,\tilde{\gamma}_n}\right)=\int_0^1\left\langle E'(P_{\tilde{\lambda}_n,\tilde{b}_n,\tilde{\gamma}_n}+\tau \tilde{\varepsilon}_{\tilde{\lambda}_n,\tilde{b}_n,\tilde{\gamma}_n}),\tilde{\varepsilon}_{\tilde{\lambda}_n,\tilde{b}_n,\tilde{\gamma}_n}\right\rangle d\tau
\]
and $E'(w)=-\Delta w-|w|^\frac{4}{N}w-|y|^{-2\sigma}w$, we have
\[
E\left(u_n\right)-E\left(P_{\tilde{\lambda}_n,\tilde{b}_n,\tilde{\gamma}_n}\right)=O\left(\frac{1}{{\tilde{\lambda}_n}^2}\|\tilde{\varepsilon}_n\|_{H^1}\right)=O\left(|t|^\frac{\alpha K-4}{4-\alpha}\right).
\]
Similarly, we have
\[
E\left(u_\infty\right)-E\left(P_{\tilde{\lambda}_\infty,\tilde{b}_\infty,\tilde{\gamma}_\infty}\right)=O\left(\frac{1}{{\tilde{\lambda}_\infty}^2}\|\tilde{\varepsilon}_\infty\|_{H^1}\right)=O\left(|t|^\frac{\alpha K-4}{4-\alpha}\right).
\]
From the continuity of $E$, we have
\[
\lim_{n\rightarrow \infty}E\left(P_{\tilde{\lambda}_n,\tilde{b}_n,\tilde{\gamma}_n}\right)=E\left(P_{\tilde{\lambda}_\infty,\tilde{b}_\infty,\tilde{\gamma}_\infty}\right)
\]
and from the conservation of energy,
\[
E\left(u_n\right)=E\left(u_n(t_n)\right)=E\left(P_{\tilde{\lambda}_{1,n},\tilde{b}_{1,n},\tilde{\gamma}_{1,n}}\right)=E_0.
\]
Therefore, we have
\[
E\left(u_\infty\right)=E_0+o_{t\nearrow0}(1)
\]
and since $E\left(u_\infty\right)$ is constant for $t$, $E\left(u_\infty\right)=E_0$.
\end{proof}

\appendix
\section{Proof of Theorem \ref{theorem:ESCBS}}
In this section, we describe the proof of Theorem \ref{theorem:ESCBS}.

\label{ProofESCBS}
\begin{proof}[Proof of Theorem \ref{theorem:ESCBS}]
We assume that $u$ is a critical-mass radial solution of (NLS$-$) and blows up at $T^*$. Let a sequence $(t_n)_{n\in\mathbb{N}}$ be such that $t_n\rightarrow T^*$ as $n\rightarrow T^*$ and define
\[
\lambda_n:=\frac{\|\nabla Q\|_2}{\|\nabla u(t_n)\|},\quad v_n(x):={\lambda_n}^\frac{N}{2}u(t_n,\lambda_n x).
\]
Then,
\[
\|v_n\|_2=\|Q\|_2,\quad \|\nabla v_n\|_2=\|\nabla Q\|_2
\]
hold. Moreover,
\[
E_0:=E(u(t_n))\geq E_{\mathrm{crit}}(u(t_n))=\frac{E(v_n)}{{\lambda_n}^2}.
\]
Therefore, we obtain
\[
\limsup_{n\rightarrow\infty}E(v_n)\leq 0.
\]
From the standard concentration argument (see \cite{MRUPB,LMR}), there exist sequences $(x_n)_{n\in\mathbb{N}}\subset\mathbb{R}^N$ and $(\gamma_n)_{n\in\mathbb{N}}\subset\mathbb{R}$ such that
\[
v_n(\cdot-x_n)e^{i\gamma_n}\rightarrow Q\mbox{ in }H^1(\mathbb{R}^N)\quad(n\rightarrow\infty).
\]
Moreover, up to a subsequence, we have
\[
v_ne^{i\gamma_n}\rightarrow Q\mbox{ in }H^1(\mathbb{R}^N)\quad(n\rightarrow\infty).
\]
Indeed, if $(x_n)_{n\in\mathbb{N}}$ is unbounded, we may assume $x_n\rightarrow\infty$ as $n\rightarrow\infty$. Then, since $v_n$ decay uniformly by the radial lemma, we have
\[
0=\lim_{n\rightarrow\infty}\left\|v_n(\cdot-x_n)e^{i\gamma_n}-Q\right\|_{H^1}^2=2\left\|Q\right\|_{H^1}^2-\lim_{n\rightarrow\infty}2\left(v_n(\cdot-x_n)e^{i\gamma_n},Q\right)_{H^1}=2\left\|Q\right\|_{H^1}^2.
\]
It is a contradiction. Therefore, $(x_n)_{n\in\mathbb{N}}$ is bounded. We may assume that $(x_n)_{n\in\mathbb{N}}$ is a convergent sequence. Let define $x_0:=\lim_{n\rightarrow\infty}x_n$. Then, we have
\[
v_ne^{i\gamma_n}\rightarrow Q(\cdot+x_0)\quad\mbox{in }H^1(\mathbb{R}^N)\quad(n\rightarrow\infty).
\]
Since $v_n$ and $Q$ are radial, we obtain $x_0=0$.

Here, we have
\[
\left\||\cdot|^{-\sigma}u(t_n)\right\|_2^2=\frac{\left\||\cdot|^{-\sigma}v_n\right\|_2^2}{{\lambda_n}^{2\sigma}}.
\]
Therefore, since $E_{\mathrm{crit}}(u)\geq 0$,
\[
E_0=E(u(t_n))\geq \frac{\left\||\cdot|^{-\sigma}v_n\right\|_2^2}{{\lambda_n}^{2\sigma}}\rightarrow \infty\quad(n\rightarrow\infty).
\]
It is a contradiction.
\end{proof}

\section{Solutions for $(S_{j,k})$}
\label{sec:solofs}
In this section, we construct solutions $(P_{j,k}^+,P_{j,k}^-,\beta_{j,k},c_{j,k}^+)\in{\mathcal{Y}'}^2\times\mathbb{R}^2$ for systems $(S_{j,k})$ in Proposition \ref{theorem:constprof}.

\begin{proposition}
\label{existsol1}
For any $g\in H^{-1}(\mathbb{R}^N)$ such that $\left\langle g,\frac{\partial Q}{\partial x_j}\right\rangle=0\ (j=1,\dots,N)$, there exists $f\in H^1(\mathbb{R}^N)$ such that $L_+f=g$ in $H^{-1}$. Similarly, for any $g\in H^{-1}(\mathbb{R}^N)$ such that $\left\langle g,Q\right\rangle=0$, there exists $f\in H^1(\mathbb{R}^N)$ such that $L_-f=g$ in $H^{-1}$.
\end{proposition}

\begin{proof}
Let $\phi_+$ be the ground state of $L_+$ and $\mu_+$ be the eigenvalue of $\phi_+$. Then, $\mu_+<0$ and we may assume $\|\phi_+\|_2=1$. Let define $H_\pm$ which is subspaces of $H^1(\mathbb{R}^N)$ by
\[
H_+:=\Span\left\{\phi_+,\frac{\partial Q}{\partial x_1},\dots,\frac{\partial Q}{\partial x_N}\right\}^\bot,\quad H_-:=\Span\left\{Q\right\}^\bot,
\]
then $H_\pm$ is Hilbert space and
\[
\exists C_\pm>0\forall f\in H_\pm,\ \langle L_\pm f,f\rangle\geq C_\pm \|f\|_{H^1}^2
\]
hold, where double sign correspond. Therefore, from the Lax-Milgram theorem,
\[
\forall g\in H_\pm^*\exists!\tilde{f}_\pm\in H_\pm,\ L_\pm \tilde{f}_\pm=g\ \mbox{in}\ H_\pm^*
\]
hold, where double sign correspond.

Here, let $\left\langle g,\frac{\partial Q}{\partial x_j}\right\rangle=0$, $f:=\tilde{f}+\frac{\langle g,\phi_+\rangle}{{\mu_+}^2}\phi_+$, and $\tilde{\varphi}:=\varphi-(\varphi,\phi_+)\phi_+-(\varphi,\nabla Q)\cdot\nabla Q$ for each $\varphi\in H^1(\mathbb{R}^N)$. Then, $\tilde{\varphi}\in H_+$ and we have
\begin{align*}
\left\langle L_+f,\varphi\right\rangle=&\left\langle f,L_+\varphi\right\rangle=\left\langle f,L_+\tilde{\varphi}+\mu_+(\varphi,\phi_+)\phi_+\right\rangle=\left\langle L_+\tilde{f},\tilde{\varphi}\right\rangle+\left\langle \frac{\langle g,\phi_+\rangle}{\mu_+}\phi_+,\mu_+(\varphi,\phi_+)\phi_+\right\rangle\\
=&\left\langle g,\tilde{\varphi}\right\rangle+(\varphi,\phi_+)\left\langle g,\phi_+\right\rangle+(\varphi,\nabla Q)\cdot\langle g,\nabla Q\rangle\\
=&\langle g,\varphi\rangle.
\end{align*}
This means that $L_+f=g$ in $H^{-1}$.

The same is proved in the case of $\left\langle g,Q\right\rangle=0$.
\end{proof}

\begin{proposition}
For any $g,h\in\mathcal{Y}$, there exists $f\in\mathcal{Y}$ such that $L_+f=g+\frac{1}{|y|^{2\sigma}}h$. Similarly, for any $g,h\in\mathcal{Y}$ such that $\left\langle g+\frac{1}{|y|^{2\sigma}}h,Q\right\rangle=0$, there exists $f\in\mathcal{Y}$ such that $L_-f=g+\frac{1}{|y|^{2\sigma}}h$.
\end{proposition}

\begin{proof}
We prove only for $L_+$. Since $\mathcal{Y}\subset H^1_{\mathrm{rad}}(\mathbb{R}^N)$, the existence of $H^1$-solution is clearly from Proposition \ref{existsol1}.

Firstly, based on a classical argument of elliptic partial differential equations, we have $f\in C^\infty(\mathbb{R}^N\setminus\{0\})$. From the maximum principal,
\[
\forall \alpha\in{\mathbb{N}_0}^N\exists C_\alpha,\kappa_\alpha>0,\ |x|\geq 1\Rightarrow \left|\left(\frac{\partial}{\partial x}\right)^\alpha f(x)\right|\leq C_\alpha(1+|x|^{\kappa_\alpha})Q(x)
\]
holds. Since $g+\frac{1}{|y|^{2\sigma}}h\in L^p(\mathbb{R}^N)$ for some $p>\max\{\frac{N}{2},1\}$, we have $f\in L^\infty(\mathbb{R}^N)$ (see \cite{GT}). Furthermore, since
\[
-\Delta f+f=\left(1+\frac{4}{N}\right)Q^{\frac{4}{N}}f+g+\frac{1}{|y|^{2\sigma}}h\in L^p(\mathbb{R}^N),
\]
we have $f\in W^{2,p}(\mathbb{R}^N)\hookrightarrow C^{0,\gamma}(\mathbb{R}^N)$ for some $\gamma\in(0,1)$. Namely, $f\in\mathcal{Y}$.
\end{proof}

\begin{proposition}
\label{theorem:Ssol}
The system $(S_{j,k})$ has a solution $(P_{j,k}^+,P_{j,k}^-,\beta_{j,k},c_{j,k}^+)\in\mathcal{Y}^2\times\mathbb{R}^2$.
\end{proposition}

\begin{proof}
We solve
\begin{empheq}[left={(S_{j,k})\ \empheqlbrace\ }]{align*}
&L_+P_{j,k}^+-F_{j,k}^+-\beta_{j,k}\frac{|y|^2}{4}Q-\frac{1}{|y|^{2\sigma}}F_{j,k}^{\sigma,+}-c_{j,k}^+Q=0,\\
&L_-P_{j,k}^--F_{j,k}^-+((k+1)\alpha+2j)P_{j,k}^+-\frac{1}{|y|^{2\sigma}}F_{j,k}^{\sigma,-}=0.
\end{empheq}
For $(S_{j,k})$, we consider the following two systems:
\begin{empheq}[left={(\tilde{S}_{j,k})\ \empheqlbrace\ }]{align*}
&L_+\tilde{P}_{j,k}^+-F_{j,k}^+-\beta_{j,k}\frac{|y|^2}{4}Q-\frac{1}{|y|^{2\sigma}}F_{j,k}^{\sigma,+}=0,\\
&L_-\tilde{P}_{j,k}^--F_{j,k}^-+((k+1)\alpha+2j)\tilde{P}_{j,k}^+-\frac{1}{|y|^{2\sigma}}F_{j,k}^{\sigma,-}=0.
\end{empheq}
and
\begin{empheq}[left={(S'_{j,k})\ \empheqlbrace\ }]{align*}
&P_{j,k}^+=\tilde{P}_{j,k}^+-\frac{c_{j,k}^+}{2}\Lambda Q,\\
&P_{j,k}^-=\tilde{P}_{j,k}^--c_{j,k}^-Q-\frac{((k+1)\alpha+2j)c_{j,k}^+}{8}|y|^2Q.
\end{empheq}
Then, by applying $(S'_{j,k})$ to a solution for $(\tilde{S}_{j,k})$, we obtain a solution for $(S_{j,k})$.

Firstly, we solve
\begin{empheq}[left={(\tilde{S}_{0,0})\ \empheqlbrace\ }]{align*}
&L_+\tilde{P}_{0,0}^+-\beta_{0,0}\frac{|y|^2}{4}Q-\frac{1}{|y|^{2\sigma}}Q=0,\\
&L_-\tilde{P}_{0,0}^-+\alpha \tilde{P}_{0,0}^+=0.
\end{empheq}
For any $\beta_{0,0}\in\mathbb{R}$, there exists a solution $\tilde{P}_{0,0}^+\in\mathcal{Y}$. Let
\[
\beta_{0,0}:=\frac{4\sigma\||\cdot|^{-\sigma}Q\|_2^2}{\||\cdot|Q\|_2^2}.
\]
Then, since
\[
\left(\tilde{P}_{0,0}^+,Q\right)_2=-\frac{1}{2}\left\langle L_+\tilde{P}_{0,0}^+,\Lambda Q\right\rangle=-\frac{1}{2}\left\langle\beta_{0,0}\frac{|y|^2}{4}Q+\frac{1}{|y|^{2\sigma}}Q,\Lambda Q\right\rangle=\frac{1}{2}\left(\frac{\beta_{0,0}}{4}\||\cdot|Q\|_2^2-\sigma\||\cdot|^{-\sigma}Q\|_2^2\right)=0,
\]
there exists a solution $\tilde{P}_{0,0}^-\in\mathcal{Y}$. By taking $c_{0,0}^+=0$, we obtain a solution $(P_{0,0}^+,P_{0,0}^-,\beta_{0,0},c_{0,0}^+)\in\mathcal{Y}^2\times\mathbb{R}^2$ for $(S_{0,0})$. Here, let $H(j_0,k_0)$ denote by that
\[
\forall (j,k)\in\Sigma_{K+K'},\ k<k_0\ \mbox{or}\ (k=k_0\ \mbox{and}\ j<j_0) \Rightarrow (S_{j,k})\ \mbox{has a solution}\ (P_{j,k}^+,P_{j,k}^-,\beta_{j,k},c_{j,k}^+)\in\mathcal{Y}^2\times\mathbb{R}^2.
\]
From the above discuss, $H(1,0)$ is true. If $H(j_0,k_0)$ is true, then $F_{j_0,k_0}^\pm$ is defined and belongs to $\mathcal{Y}$. Moreover, for any $\beta_{j_0,k_0}$, there exists a solution $\tilde{P}_{j_0,k_0}^+$. Let be $\beta_{j_0,k_0}$ such that
\[
\left\langle-F_{j,k}^-+((k+1)\alpha+2j)\tilde{P}_{j,k}^+-\frac{1}{|y|^{2\sigma}}F_{j,k}^{\sigma,-},Q\right\rangle=0.
\]
Then, we obtain a solution $\tilde{P}_{j_0,k_0}^-$. Here, we define
\begin{align*}
c_{j_0,k_0}^-&:=\left\{
\begin{array}{cc}
\frac{\tilde{P}_{j_0,k_0}^-(0)}{Q(0)}&(j_0+k_0\neq K+1),\\
0&(j_0+k_0=K+1,\mbox{\ and\ }\tilde{P}_{j_0,k_0}^-(0)\neq 0),\\
1&(j_0+k_0=K+1,\mbox{\ and\ } \tilde{P}_{j_0,k_0}^-(0)=0),
\end{array}\right.\\
c_{j_0,k_0}^+&:=\left\{
\begin{array}{cc}
0&(j_0+k_0\leq K),\\
0&(j_0+k_0=K,\mbox{\ and\ }\tilde{P}_{j_0,k_0}^+(0)\neq 0),\\
1&(j_0+k_0=K,\mbox{\ and\ }\tilde{P}_{j_0,k_0}^+(0)=0),\\
\frac{2\tilde{P}_{j_0,k_0}^+(0)}{Q(0)}&(j_0+k_0\geq K+2).
\end{array}\right.
\end{align*}
Then, we obtain a solution for $(S_{j_0,k_0})$. This means that $H(j_0+1,k_0)$ is true if $j_0+k_0\leq K+K'-1$ and $H(0,k_0+1)$ is true if $j_0+k_0=K+K'$. In particular, $H(0,K+K'+1)$ means that for any $(j,k)\in\Sigma_{K+K'}$, there exists a solution $(P_{j,k}^+,P_{j,k}^-,\beta_{j,k},c_{j,k}^+)\in\mathcal{Y}^2\times\mathbb{R}^2$.

Furthermore,  $P_{j,k}^\pm(0)\neq 0$ for $j+k=K+1$ and $P_{j,k}^\pm(0)=0$ for $j+k\geq K+2$ hold.
\end{proof}

\begin{proposition}
\label{pconti}
For $P_{j,k}^\pm$, $\Lambda P_{j,k}^\pm\in H^1(\mathbb{R}^N)\cap C(\mathbb{R}^N)$. Namely, $P_{j,k}^\pm\in\mathcal{Y}'$.
\end{proposition}

\begin{proof}
Regarding $\Lambda P_{j,k}^\pm\in H^1(\mathbb{R}^N)$, proving $y_lP_{j,k}^\pm\in H^2(\mathbb{R}^N)$ is sufficient. Since
\[
L_+(y_lP_{j,k}^+)=y_lF_{j,k}^++\beta_{j,k}\frac{|y|^2y_l}{4}Q+\frac{y_l}{|y|^{2\sigma}}F_{j,k}^{\sigma,+}+c_{j,k}^+y_lQ
\]
and
\[
\left|\frac{y_l}{|y|^{2\sigma}}F_{j,k}^{\sigma,+}\right|\leq \frac{1}{|y|^{2\sigma-1}}|F_{j,k}^{\sigma,+}|\in L^2(\mathbb{R^N}),
\]
we have $y_lP_{j,k}^+\in H^2(\mathbb{R}^N)$. Similarly, we have $y_lP_{j,k}^-\in H^2(\mathbb{R}^N)$.

Regarding $\Lambda P_{j,k}^\pm\in C(\mathbb{R}^N)$, proving $y\cdot\nabla P_{j,k}^\pm\in C(\mathbb{R}^N)$ is sufficient. Firstly,
\begin{align*}
L_+(y\cdot\nabla P_{j,k}^+)&=y\cdot\nabla(F_{j,k}^++\beta_{j,k}\frac{|y|^2}{4}Q+\frac{1}{|y|^{2\sigma}}F_{j,k}^{\sigma,+}+c_{j,k}^+Q)+2(F_{j,k}^++\beta_{j,k}\frac{|y|^2}{4}Q+\frac{1}{|y|^{2\sigma}}F_{j,k}^{\sigma,+}+c_{j,k}^+Q)\\
&\hspace{30pt}-2P_{j,k}^+2\left(\frac{4}{N}+1\right)Q^{\frac{4}{N}}P_{j,k}^+-\frac{4}{N}\left(\frac{4}{N}+1\right)Q^{\frac{4}{N}-1}y\cdot\nabla QP_{j,k}^+
\end{align*}
holds. Since $\frac{y}{|y|^{2\sigma}}\cdot\nabla F_{j,k}^{\sigma,+}\in L^p(\mathbb{R}^N)$ for some $p>\max\{\frac{N}{2},1\}$, we have $L_+(y\cdot\nabla P_{j,k}^+)\in L^p(\mathbb{R}^N)$. Therefore, we have $y\cdot\nabla P_{j,k}^+\in C(\mathbb{R}^N)$. Similarly, we have $y\cdot\nabla P_{j,k}^-\in C(\mathbb{R}^N)$.
\end{proof}

\begin{proposition}
\label{Pint}
For $P_{0,K+K'}^{\pm}$,
\[
\frac{1}{r^2}P_{0,K+K'}^\pm,\frac{1}{r}\frac{\partial P_{0,K+K'}^\pm}{\partial r}\in L^\infty(\mathbb{R}^N),
\]
where $r=|y|$.
\end{proposition}

\begin{proof}
We prove only for $P_{0,K+K'}^+$.

Let $f_k:=P_{0,K+k}^+$ for $k\in\mathbb{N}$. Here, $f_1(0)\neq0$ and $f_k(0)=0$ for $k\geq 2$ hold. Moreover, Let
\[
F_k:=f_k-\left(1+\frac{4}{N}\right)Q^{\frac{4}{N}}f_k-F_{0,K+k}^+-\beta_{0,K+k}\frac{r^2}{4}Q-c_{0,K+k}^+Q.
\]

If $r^{-q}f_k$ converges to non-zero as $r\searrow0$ for some $q\in[0,2\sigma)$ or $r^{-q}f_k$ converges as $r\searrow0$ for some $q\geq 2\sigma$, then $r^{N-1}\frac{\partial f_{k+1}}{\partial r}$ converges to $0$ as $r\searrow 0$. Indeed, if $N=1$, then $f_k\in W^{2,p}(\mathbb{R}^N)\hookrightarrow C^1(\mathbb{R}^N)$ for some $p>1$. Therefore, since $f_k$ is an even function, $\frac{\partial f_k}{\partial r}(0)=0$ holds. On the other hand, for $N\geq 2$,
\begin{align}
\label{fkeq}
\frac{1}{r^{N-1}}\frac{\partial}{\partial r}\left(r^{N-1}\frac{\partial f_{k+1}}{\partial r}\right)=F_{k+1}-\frac{1}{r^{2\sigma}}f_k
\end{align}
holds. If $r^{-q}f_k$ converges as $r\searrow0$ for some $q\geq 2\sigma$, then $r^{-2\sigma}f_k$ is bounded. Therefore, for some sufficiently large $p$, we have $f_{k+1}\in W^{2,p}(\mathbb{R}^N)\hookrightarrow C^1(\mathbb{R}^N)$. Accordingly, $r^{N-1}\frac{\partial f_{k+1}}{\partial r}$ converges to $0$ as $r\searrow 0$. On the other hand, if $r^{-q}f_k$ converges to non-zero as $r\searrow0$ for some $q\in[0,2\sigma)$, the right hand of (\ref{fkeq}) diverge $+\infty$ or $-\infty$ as $r\searrow 0$. Therefore, $r^{N-1}\frac{\partial f_{k+1}}{\partial r}$ is increasing or decreasing as $r\searrow 0$, meaning $r^{N-1}\frac{\partial f_{k+1}}{\partial r}$ converges in $[-\infty,\infty]$. Let
\[
C:=\lim_{r\searrow0}r^{N-1}\left|\frac{\partial f_{k+1}}{\partial r}\right|.
\]
Then, for any $\epsilon>0$, there exists $r_0>0$ such that $\left|\frac{\partial f_{k+1}}{\partial r}\right|\geq(C-\epsilon)r^{-(N-1)}$ for any $r\in(0,r_0)$. On the other hand, $f_{k+1}\in W^{2,p}(\mathbb{R}^N)\hookrightarrow W^{1,N}(\mathbb{R}^N)$ for some $p>\frac{N}{2}$ and $\left|\frac{\partial f_{k+1}}{\partial r}\right|=|\nabla f_{k+1}|$. Therefore, we have
\[
\infty>\int_{B(0,r_0)}|\nabla f_{k+1}(x)|^Ndx\geq C_N\int_0^{r_0}\frac{C-\epsilon}{r^{(N-1)^2}}dr.
\]
Since $\int_0^{r_0}r^{-(N-1)^2}dr=\infty$, we obtain $C-\epsilon\leq0$. Consequently, we have $C\leq 0$, meaning $C=0$.

Let $\sigma_1:=0$ and $C_1:=f_1(0)$. Moreover, let
\[
\sigma_{k+1}:=\left\{
\begin{array}{cc}
1-\sigma+\sigma_k&\left(\sigma_k<\sigma\right)\\
1&\left(\sigma_k\geq\sigma\right)
\end{array}\right.,\quad C_{k+1}:=\left\{
\begin{array}{cc}
\frac{-C_k}{2\sigma_{k+1}(N-2(\sigma-\sigma_k))}&\left(\sigma_k<\sigma\right)\\
\frac{F_{k+1}(0)-0^{2(\sigma_k-\sigma)}C_k}{2N}&\left(\sigma_k\geq\sigma\right)
\end{array}\right..
\]
In particular, if $\sigma_k<\sigma$, then $C_k\neq 0$. Then,
\begin{eqnarray}
\label{origindecay}
\lim_{r\searrow0}\frac{1}{r^{2\sigma_k}}f_k(r)=C_k
\end{eqnarray}
holds. For $k=1$, it clearly holds. Moreover, for $k\geq 2$,
\[
\lim_{r\searrow0}\frac{1}{r^{2\sigma_k-1}}\frac{\partial f_k}{\partial r}(r)=2\sigma_kC_k
\]
holds. Indeed, if (\ref{origindecay}) holds for some $k$, then $r^{N-1}\frac{\partial f_{k+1}}{\partial r}$ converges to $0$ as $r\searrow 0$ in both cases $\sigma_k<\sigma$ and $\sigma_k\geq \sigma$ from the above discuss. We assume $\sigma_k<\sigma$. Since
\[
\frac{1}{r^{N-1}}\frac{\partial}{\partial r}\left(r^{N-1}\frac{\partial f_{k+1}}{\partial r}\right)=F_{k+1}-\frac{1}{r^{2(\sigma-\sigma_k)}}\frac{1}{r^{2\sigma_k}}f_k,
\]
for any $\epsilon>0$, there exists $r_0>0$ such that
\[
(-C_k-\epsilon)r^{N-1-2(\sigma-\sigma_k)}\leq \frac{\partial}{\partial r}\left(r^{N-1}\frac{\partial f_{k+1}}{\partial r}\right)\leq (-C_k+\epsilon)r^{N-1-2(\sigma-\sigma_k)}
\]
for any $r\in(0,r_0)$. Integrating in $[0,r]$, we have
\[
\frac{-C_k-\epsilon}{N-2(\sigma-\sigma_k)}r^{1-2(\sigma-\sigma_k)}\leq \frac{\partial f_{k+1}}{\partial r}\leq \frac{-C_k+\epsilon}{N-2(\sigma-\sigma_k)}r^{1-2(\sigma-\sigma_k)}.
\]
Integrating in $[0,r]$ again, we have
\[
\frac{-C_k-\epsilon}{(2-2(\sigma-\sigma_k))(N-2(\sigma-\sigma_k))}r^{2-2(\sigma-\sigma_k)}\leq f_{k+1}\leq \frac{-C_k-\epsilon}{(2-2(\sigma-\sigma_k))(N-2(\sigma-\sigma_k))}r^{2-2(\sigma-\sigma_k)}.
\]
Therefore, we have
\[
\lim_{r\searrow0}\frac{1}{r^{2\sigma_{k+1}-1}}\frac{\partial f_{k+1}}{\partial r}(r)=2\sigma_{k+1}C_{k+1},\quad\lim_{r\searrow0}\frac{1}{r^{2\sigma_{k+1}}}f_{k+1}(r)=C_{k+1}.
\]
On the other hand, we assume $\sigma_k\geq \sigma$. Then, for any $\epsilon>0$, there exists $r_0>0$ such that
\[
(F_{k+1}(0)-0^{2(\sigma_k-\sigma)}C_k-\epsilon)r^{N-1}\leq \frac{\partial}{\partial r}\left(r^{N-1}\frac{\partial f_{k+1}}{\partial r}\right)\leq (F_{k+1}(0)-0^{2(\sigma_k-\sigma)}C_k+\epsilon)r^{N-1}
\]
for any $r\in(0,r_0)$. Integrating in the same way as for $\sigma_k<\sigma$, we have
\[
\frac{F_{k+1}(0)-0^{2(\sigma_k-\sigma)}C_k-\epsilon}{N}r\leq \frac{\partial f_{k+1}}{\partial r}\leq \frac{F_{k+1}(0)-0^{2(\sigma_k-\sigma)}C_k+\epsilon}{N}r.
\]
Moreover, since
\[
\frac{F_{k+1}(0)-0^{2(\sigma_k-\sigma)}C_k-\epsilon}{2N}r^2\leq f_{k+1}\leq \frac{F_{k+1}(0)-0^{2(\sigma_k-\sigma)}C_k+\epsilon}{2N}r^2,
\]
we have
\[
\lim_{r\searrow0}\frac{1}{r^{2\sigma_{k+1}-1}}\frac{\partial f_{k+1}}{\partial r}(r)=2\sigma_{k+1}C_{k+1},\quad\lim_{r\searrow0}\frac{1}{r^{2\sigma_{k+1}}}f_{k+1}(r)=C_{k+1}.
\]

Consequently, we obtain Proposition \ref{Pint} if $K'$ is sufficiently large.
\end{proof}

\section{Proof of Lemma \ref{decomposition}}
\label{sec:prfdecom}
In this section, we only outline the proof of Lemma \ref{decomposition}. See \cite{N,MRUPB} for detail of the proof.

\begin{definition}
For $\lambda>0$ and $\gamma\in\mathbb{R}$, define $T_{\lambda,\gamma}:H^1(\mathbb{R}^N)\rightarrow H^1(\mathbb{R}^N)$ as
\[
T_{\lambda,\gamma}u:=\lambda^{\frac{N}{2}}u(\lambda\cdot)e^{i\gamma}.
\]
\end{definition}

\begin{definition}
Define
\[
L_\Lambda^2(\mathbb{R}^N):=\left\{u\in L^2(\mathbb{R}^N)\ \middle|\ \Lambda u\in L^2(\mathbb{R}^N)\ \right\},
\]
where $\Lambda:=\frac{N}{2}+x\cdot\nabla$.
\end{definition}

\begin{definition}
Let $X$ be a normed space. For $x\in X$ and $r>0$, we define
\[
B_X(x,r):=\left\{y\in X\ \middle|\ \|x-y\|_X<r\right\}.
\]
\end{definition}

\begin{definition}
We define $\tilde{\varepsilon}:\mathbb{R}_{>0}\times\mathbb{R}^2\times H^1(\mathbb{R}^N)\times\mathbb{R}\rightarrow H^1(\mathbb{R}^N)$, $\tilde{P}:\mathbb{R}_{>0}\times\mathbb{R}^2\rightarrow H^1(\mathbb{R}^N)$, and $S:\mathbb{R}_{>0}\times\mathbb{R}^2\times H^1(\mathbb{R}^N)\times\mathbb{R}\rightarrow \mathbb{R}^3$ as
\begin{align*}
\tilde{\varepsilon}(\tilde{\lambda},\tilde{b},\tilde{\gamma},u,l)&:=\tilde{\lambda}^{\frac{N}{2}}u(\tilde{\lambda}\cdot)e^{i\tilde{b}|\cdot|^2/4-i\tilde{\gamma}}-\tilde{P}(\tilde{\lambda},\tilde{b},l),\\
\tilde{P}(\tilde{\lambda},\tilde{b},l)&:=Q+\sum_{(j,k)\in\Sigma_{K+K'}}\left(\tilde{b}^{2j}(|l|\tilde{\lambda})^{(k+1)\alpha}P_{j,k}^++i\tilde{b}^{2j+1}(|l|\tilde{\lambda})^{(k+1)\alpha}P_{j,k}^-\right),\\
S(\tilde{\lambda},\tilde{b},\tilde{\gamma},u,l)&:=\left(\left(\tilde{\varepsilon}(\tilde{\lambda},\tilde{b},\tilde{\gamma},u,l),i\Lambda \tilde{P}(\tilde{\lambda},\tilde{b},l)\right)_2,\left(\tilde{\varepsilon}(\tilde{\lambda},\tilde{b},\tilde{\gamma},u,l),|\cdot|^2\tilde{P}(\tilde{\lambda},\tilde{b},l)\right)_2,\left(\tilde{\varepsilon}(\tilde{\lambda},\tilde{b},\tilde{\gamma},u,l),i\rho\right)_2\right),
\end{align*}
respectively.
\end{definition}

Here, $S:\mathbb{R}_{>0}\times\mathbb{R}^2\times H^1(\mathbb{R}^N)\times\mathbb{R}\rightarrow \mathbb{R}^3$ is a continuous function and $S:\mathbb{R}_{>0}\times\mathbb{R}^2\times H^1(\mathbb{R}^N)\times(\mathbb{R}\setminus\{0\})\rightarrow \mathbb{R}^3$ is a $C^1$ function.

\begin{proposition}
\label{predecomposition}
There exist $R,\overline{l},\overline{b}>0$, $\overline{\gamma}\in(0,\pi)$, $\overline{\lambda}\in(0,1)$, and a unique function $\tilde{S}:B_{H^1}(Q,R)\times(-\overline{l},\overline{l})\rightarrow(1-\overline{\lambda},1+\overline{\lambda})\times(-\overline{b},\overline{b})\times(-\overline{\gamma},\overline{\gamma})$ such that $\tilde{S}(Q,0)=(1,0,0)$ and $S(\tilde{S}(u,l),u,l)=0$ for $(u,l)\in B_{H^1}(Q,R)\times(-\overline{l},\overline{l})$. Furthermore, $\tilde{S}$ is a continuous function.
\end{proposition}

\begin{proof}
This proposition is proved by the implicit function theorem considering $\tilde{\varepsilon}(1,0,0,Q,0)=0$ (see Lemma 2 in \cite{MRUPB}).
\end{proof}

\begin{definition}
For $(u,l)\in B_{H^1}(Q,R)\times(-\overline{l},\overline{l})$, define
\[
\left(\tilde{\lambda}(u,l),\tilde{b}(u,l),\tilde{\gamma}(u,l)\right):=\tilde{S}(u,l).
\]
\end{definition}

\begin{proposition}
The function $\tilde{S}$ from Proposition \ref{predecomposition} is a $C^1$ function in $B_{H^1}(Q,R)\times(0,\overline{l})$.
\end{proposition}

\begin{proof}
If $R$, $\overline{l}$, $\overline{b}$, $\overline{\lambda}$, and $\overline{\gamma}$ are sufficiently small, then $D_{(\tilde{\lambda},\tilde{b},\tilde{\gamma})}S(u,l)$ is a regular matrix for any $(u,l)\in B_{H^1}(Q,R)\times(-\overline{l},\overline{l})$.

For any $(u_0,l_0)\in B_{H^1}(Q,R)\times(0,\overline{l})$, we have
\[
S(\tilde{\lambda}(u_0,l_0),\tilde{b}(u_0,l_0),\tilde{\gamma}(u_0,l_0),u_0,l_0)=0.
\]
Therefore, there exist $R_{u_0,l_0},\overline{l}_{u_0,l_0},\overline{b}_{u_0,l_0}>0$, $\overline{\gamma}_{u_0,l_0}\in(0,\pi)$, $\overline{\lambda}_{u_0,l_0}\in(0,1)$, and a unique function $\tilde{S}_{u_0,l_0}:B_{H^1}(Q,R_{u_0,l_0})\times(l_0-\overline{l}_{u_0,l_0},l_0+\overline{l}_{u_0,l_0})\rightarrow(1-\overline{\lambda}_{u_0,l_0},1+\overline{\lambda}_{u_0,l_0})\times(-\overline{b}_{u_0,l_0},\overline{b}_{u_0,l_0})\times(-\overline{\gamma}_{u_0,l_0},\overline{\gamma}_{u_0,l_0})$ such that
\[
\tilde{S}_{u_0,l_0}(u_0,l_0)=\tilde{S}(u_0,l_0),\quad S(\tilde{S}_{u_0,l_0}(u,l),u,l)=0\ \mbox{for any}\ (u,l)\in B_{H^1}(Q,R_{u_0,l_0})\times(l_0-\overline{l}_{u_0,l_0},l_0+\overline{l}_{u_0,l_0}).
\]
Moreover, $\tilde{S}_{u_0,l_0}$ is a $C^1$ function. According to the uniqueness,  $\tilde{S}=\tilde{S}_{u_0,l_0}$ holds in a neighbourhood of $(u_0,l_0)$.
\end{proof}

\begin{definition}
For any $l,\delta>0$, we define
\[
U_{l,\delta}:=\left\{u\in H^1(\mathbb{R}^N)\ \bigg|\ \inf_{\lambda\in(0,l),\gamma\in\mathbb{R}}\|\lambda^{\frac{N}{2}}u(t,\lambda \cdot)e^{i\gamma}-Q\|_{H^1}< \delta\right\}.
\]
\end{definition}

\begin{proposition}
For any $\delta$ such that is sufficiently small, the domain of $\tilde{\lambda}$, $\tilde{b}$, and $\tilde{\gamma}$ are extended to $U_{\overline{l},\delta}$. This extension is a unique and $\tilde{\gamma}$ is a polyvalent function.
\end{proposition}

\begin{proof}
For any $u\in U_{\overline{l},\delta}$, there exist $l\in(0,\overline{l})$ and $\gamma\in\mathbb{R}$ such that $T_{l,\gamma}u\in B(Q,\delta)$. Then, we define the extension as
\[
\tilde{\lambda}(u):=l\tilde{\lambda}(T_{l,\gamma}u,l),\quad \tilde{b}(u):=\tilde{b}(T_{l,\gamma}u,l),\quad \tilde{\gamma}(u):=\tilde{\lambda}(T_{l,\gamma}u,l)-\gamma.
\]
See \cite{N} for well-definedness and the uniqueness.
\end{proof}


\begin{thebibliography}{99}
\bibitem{BLGS} H. Berestycki and P.-L. Lions. Nonlinear scalar field equations. I. Existence of a ground state. \textit{Arch. Rational Mech. Anal.} 82 (1983), no. 4, 313–345.
\bibitem{C} R. Carles. Nonlinear Schr\"{o}dinger equations with repulsive harmonic potential and applications. \textit{SIAM J. Math. Anal.} 35 (2003), no. 4, 823--843.
\bibitem{CN} R. Carles and Y. Nakamura. Nonlinear Schr\"{o}dinger equations with Stark potential. \textit{Hokkaido Math. J.} 33 (2004), no. 3, 719--729.
\bibitem{CSSE} T. Cazenave. \textit{Semilinear Schr\"{o}dinger equations}. Courant Lecture Notes in Mathematics, 10. New York University, Courant Institute of Mathematical Sciences, New York; American Mathematical Society, Providence, RI, 2003.
\bibitem{EF} E. Csobo and F. Genoud. Minimal mass blow-up solutions for the $L^2$ critical NLS with inverse-square potential. \textit{Nonlinear Anal}. 168 (2018), 110–129.
\bibitem{GT} D. Gilbarg and N. S. Trudinger. \textit{Elliptic partial differential equations of second order. Second edition}. Grundlehren der Mathematischen Wissenschaften [Fundamental Principles of Mathematical Sciences], 224. Springer-Verlag, Berlin, 1983.
\bibitem{KGS} M. K. Kwong. Uniqueness of positive solutions of $\delta u-u+u^p=0\ \mathrm{in}\ \mathbb{R}^n$. \textit{Arch. Rational Mech. Anal.} 105 (1989), no. 3, 243–266.
\bibitem{LMR} S. Le Coz, Y. Martel and P. Rapha\"{e}l. Minimal mass blow up solutions for a double power nonlinear Schr\"{o}dinger equation. \textit{Rev. Mat. Iberoam.} 32 (2016), no. 3, 795–833.
\bibitem{N} N. Matsui. Minimal mass blow-up solutions for nonlinear Schr\"{o}dinger equations with a potential, arXiv preprint \url{https://arxiv.org/abs/2007.15968}
\bibitem{MMMB} F. Merle. Determination of blow-up solutions with minimal mass for nonlinear Schr\"{o}dinger equations with critical power. \textit{Duke Math. J.} 69 (1993), no. 2, 427–454.
\bibitem{MRO} F. Merle and P. Raphael. On universality of blow-up profile for $L^2$ critical nonlinear Schr\"{o}dinger equation. \textit{Invent. Math.} 156 (2004), no. 3, 565–672.
\bibitem{MRUPB} F. Merle and P. Raphael. The blow-up dynamic and upper bound on the blow-up rate for critical nonlinear Schr\"{o}dinger equation. \textit{Ann. of Math.} (2) 161 (2005), no. 1, 157–222.
\bibitem{MRUDB} F. Merle and P. Raphael. On a sharp lower bound on the blow-up rate for the $L^2$ critical nonlinear Schr\"{o}dinger equation. \textit{J. Amer. Math. Soc.} 19 (2006), no. 1, 37–90.
\bibitem{RSEU} P. Rapha\"{e}l and J. Szeftel. Existence and uniqueness of minimal blow-up solutions to an inhomogeneous mass critical NLS. \textit{J. Amer. Math. Soc.} 24 (2011), no. 2, 471–546.
\bibitem{WL} M. Weinstein. Lyapunov stability of ground states of nonlinear dispersive evolution equations. \textit{Comm. Pure Appl. Math.} 39 (1986), no. 1, 51–67.
\bibitem{WGS} M. Weinstein. Nonlinear Schr\"{o}dinger equations and sharp interpolation estimates. \textit{Comm. Math. Phys.} 87 (1982/83), no. 4, 567–576.
\end{thebibliography}
\end{document}